\documentclass{amsart}

\usepackage{amsmath}
\usepackage{amsfonts}
\usepackage{amssymb}
\usepackage{amscd}
\usepackage{amsthm}
\usepackage[mathscr]{eucal}
\theoremstyle{plain}
\theoremstyle{definition}
\newtheorem{theorem}{Theorem}[section]
\newtheorem{lemma}[theorem]{Lemma}
\newtheorem{proposition}[theorem]{Proposition}
\newtheorem{corollary}[theorem]{Corollary}
\newtheorem{definition}[theorem]{Definition}
\newtheorem{remark}[theorem]{Remark}

\newtheorem{notation}[theorem]{Notation}
\newtheorem*{claim}{Claim}
\newtheorem*{claim1}{Claim 1}
\newtheorem*{claim2}{Claim 2}

\newcommand{\mc}[1]{{\mathcal #1}}

\newcommand{\seq}[2]{\ensuremath{({#1}_{#2})_{#2\in\mathbb{N}}}}
 \DeclareMathOperator{\supp}{supp}

\DeclareMathOperator{\spann}{span}
 \DeclareMathOperator{\ran}{ran}

\DeclareMathOperator{\sgn}{sgn} 
\DeclareMathOperator{\diag}{diag}
\newcommand{\R}{\mathbb{R}}
\newcommand{\N}{\mathbb{N}}
\newcommand{\Q}{\mathbb{Q}}
\newcommand{\C}{\mathbb{C}}
\newcommand{\HH}{\mathbb{H}}
\newcommand{\e}{\varepsilon}

\newcommand{\eqs}{{\mathfrak X}}

\theoremstyle{plain}

\begin{document}

\title{Strictly singular non-compact diagonal operators on HI spaces}

\author{Spiros A. Argyros}
\address[S.A. Argyros]{Department of Mathematics,
National Technical University of Athens}
\email{sargyros@math.ntua.gr}

\author{Irene  Deliyanni}
\address[I. Deliyanni]{Neapoleos 18,  15341, Athens, Greece}
 \email{ideliyanni@yahoo.gr}

\author{Andreas G. Tolias}
\address[A. Tolias]{Department of Mathematics,
 University of the Aegean}
\email{atolias@math.aegean.gr}

 \begin{abstract}
 We construct a Hereditarily Indecomposable
 Banach space $\eqs_d$ with a Schauder basis \seq{e}{n}
 on which there exist strictly singular non-compact diagonal operators.
 Moreover, the space
 $\mc{L}_{\diag}(\eqs_d)$ of diagonal operators with respect
 to the basis \seq{e}{n} contains an isomorphic copy of $\ell_{\infty}(\N)$.
\end{abstract}



\maketitle

\tableofcontents

 \section{Introduction}
 In the present paper we study the structure of the diagonal operators on Hereditarily
 Indecomposable spaces
 having a Schauder basis. The class of Hereditarily Indecomposable (HI) Banach spaces
 was
 introduced in the early 90's by W.T. Gowers and B. Maurey \cite{GM1}
 and led to the solution of many long standing open problems in  Banach space theory.
 Since then the class of HI Banach spaces,
 as well as the spaces
 of  bounded linear  operators acting  on them have been studied
 extensively.

 We begin by recalling  that
 an infinite dimensional Banach space
 $X$ is HI provided no closed  subspace $Y$ of $X$
 is of the form $Y=Z\oplus W$ with both $Z, W$ being of infinite
 dimension.
 For a Banach space $X$
 we shall use $\mc{L}(X)$ to denote the space  of bounded linear operators  $T:X\to X$,
 while the notation  $\mc{S}(X)$, $\mc{K}(X)$
 will stand for the ideals of  strictly singular
 and  compact operators on $X$ respectively.
   As was shown by  Gowers and  Maurey (\cite{GM1}),
  for a complex HI space $X$,
    every
  $T\in \mc{L}(X)$ takes the form $T=\lambda I+S$
  with $\lambda\in\C$ and $S\in \mc{S}(X)$ (by $I$ we shall always denote the identity operator).
  However, it is not true in general,  that each $T\in\mc{L}(X)$,
  for  a
  real HI Banach space $X$,
   can be written as
  $T=\lambda I+S$ with $\lambda\in \R$ and $S\in S(X)$;
 although     this happens   for the
  space $X_{GM}$ of Gowers and Maurey \cite{GM1}
  and for the asymptotic $\ell_1$ HI space $X_{AD}$ constructed by Argyros and Deliyanni \cite{AD}
  (for a proof see e.g. \cite{AT1}).
  V. Ferenczi proved (\cite{Fe1}) that for every real HI space $X$, the quotient space $\mc{L}(X)/\mc{S}(X)$,
  is a division algebra isomorphic to  $\R$, to $\C$,  or to the quaternionic  algebra
  $\HH$.
  V. Ferenczi \cite{Fe3} has presented two real HI Banach spaces
     $X_{\C}$ and $X_{\HH}$ with $\mc{L}(X_{\C})/\mc{S}(X_{\C})$  isomorphic
  to $\C$ and $\mc{L}(X_{\HH})/\mc{S}(X_{\HH})$  isomorphic
  to $\HH$.
  A variety of spaces $X$ with a prescribed algebra
   $\mc{L}(X)/\mc{S}(X)$  were provided by  Gowers and  Maurey in \cite{GM2}.
   Although these spaces $X$
   are not HI, they do not contain any unconditional basic sequence, hence,
   Gowers' dichotomy (\cite{G2}, \cite{G3}) yields  that they
  are HI saturated.
  Argyros and Manoussakis \cite{AM}, provided  an unconditionally
  saturated Banach  space $X$ with the property that every  $T\in\mc{L}(X)$
  is of the form $T=\lambda I+S$ with $S\in \mc{S}(X)$.

  The problem of the existence of strictly singular non-compact
  operators on HI spaces has been studied by several authors.
   The first result in this
   direction, due to
     Gowers (\cite{G1}), is an operator $T:Y\to X_{GM}$, for some subspace
   $Y$ of the Gowers-Maurey space $X_{GM}$, such that $T$ is not of the form $T=\lambda i_{Y,X}+K$
   with $K$ compact, where $i_{Y,X}$ is the canonical injection from $Y$ into $X$.
   Several extensions of the above result have been given in \cite{AnOdScTo},
   \cite{AnSa} and \cite{Sc2}.

     Argyros and Felouzis (\cite{AF}) using interpolation methods, provided
   examples of HI spaces on which there do exist strictly singular non-compact
   operators.
      G. Androulakis and Th. Schlumprecht \cite{AnSc} constructed a strictly singular
   non-compact operator    $T:X_{GM}\to X_{GM}$, while
   G. Gasparis \cite{Ga}, constructed strictly singular non-compact operators in the
   reflexive
   asymptotic $\ell_1$ HI space $X_{AD}$ of Argyros and Deliyanni.
   K. Beanland has extended Gasparis' result in the class of
   asymptotic $\ell_p$ HI spaces, for $1<p<\infty$, in \cite{B}.

   The structure of $\mc{L}(X)$ has been also studied for
   non-reflexive HI spaces (\cite{AT1}, \cite{AAT}, \cite{PR}).
   It is notable that in all these examples, each strictly
   singular operator $T\in\mc{L}(X)$ is  a weakly compact
   one. It is an open problem whether there exists an HI
   Banach space $X$ and $T\in\mc{L}(X)$ which is strictly singular
   and not weakly compact.

   The scalar plus compact problem was recently solved
    by S.  Argyros  and  R. Haydon \cite{AH}.
    It is shown that there exists an HI $\ell_1$ predual Banach space
    $\eqs_K$ such that every  $T:\eqs_K\to\eqs_K$ is of the form
   $T=\lambda I+K$, with
   $K$ a compact operator. The corresponding problem for reflexive
   spaces remains open.

   The present paper is devoted to the study of the subalgebra of
   diagonal operators of a HI space $X$ with a  Schauder basis \seq{e}{n}.
   Let' s  recall that  for a Banach space $X$ with an a priori fixed
    basis \seq{e}{n}, a bounded linear operator
   $T:X\to X$ is said to be diagonal, if for each $n$, $Te_n$ is a
   scalar multiple of $e_n$, $Te_n=\lambda_n e_n$. We denote by
   $\mc{L}_{\diag}(X)$ the space of all  diagonal
   operators $T:X\to X$. Note that if the diagonal operator $T$ is
   strictly singular then the sequence $(\lambda_n)_{n\in\N}$ of  eigenvalues of $T$
    converges to $0$.

     As is well known,
    when the basis \seq{e}{n} of the space $X$ is an  unconditional one,
    the space $\mc{L}_{\diag}(X)$ is isomorphic to
    $\ell_{\infty}(\N)$ and operator $T\in\mc{L}_{\diag}(X)$ is
    strictly singular if and only if $T$ is compact and this happens
     if and only if
    the sequence $(\lambda_n)_{n\in\N}$ of  eigenvalues of $T$
    is a null sequence.

   The following  question arises naturally.
   \begin{enumerate}
   \item[{\bf (Q)}]  Do  there exist
    strictly singular non-compact diagonal operators on some HI space with a Schauder basis?
   \end{enumerate}

   The aim of the present paper is to give a positive answer to
   ${\bf (Q)}$,
       by defining a HI space $\eqs_d$
   with a basis,
   on which there exist strictly singular non-compact diagonal operators.
   More precisely the space $\mc{L}_{\diag}(\eqs_d)$ contains  isomorphic copies
   of $\ell_{\infty}(\N)$ in a natural manner.

   It is worth pointing out that the construction of strictly
   singular non-compact diagonal
     operators lies heavily on the
   conditional structure of the underlying space $\eqs_d$.
   Previous constructions, like \cite{AnSc}, \cite{Ga}, concern
   the existence of strictly singular non-compact operators acting
   on the unconditional frame of the HI spaces. In particular
   Gasparis (\cite{Ga}) based his construction on an elegant idea which
   allowed him to define a mixed Tsirelson space
   $T[(\mc{S}_{n_j},\frac{1}{m_j})_j]$ such that its dual
   $T^*[(\mc{S}_{n_j},\frac{1}{m_j})_j]$ admits a $c_0^\omega$
   spreading model. An adaptation of Gasparis method in the frame
   of $T[(\mc{A}_{n_j},\frac{1}{m_j})_j]$ is the first of the
   fundamental ingredients in our construction.
   More precisely,  for an appropriate double sequence
   $(m_j,n_j)_j$, it it shown that
   the dual space $T^*[(\mc{A}_{n_j},\frac{1}{m_j})_j]$
  admits a $c_0$ spreading model. It is not known whether each
  mixed Tsirelson space of the form
  $T[(\mc{A}_{n_j},\frac{1}{m_j})_j]$
  not containing any $\ell_p(\N)$ or $c_0(\N)$,
  shares the aforementioned property. This problem
  remains open even for Schlumprecht' s space
  $S=T[(\mc{A}_n,\frac{1}{\log_2(n+1)})_{n}]$ (\cite{Sc}).
 As follows from  \cite{KL}, the space $S$ admits a $\ell_1$
 spreading model. This, however,  does not guarantee the existence
 of a $c_0$ spreading model in $S^*$.

  The second ingredient of our construction, is the finite block
     representability of the space $J_{T_0}$ in every block
     subspace of $\eqs_d$.  The space $J_{T_0}$,  defined
     in \cite{ALT2}, has a Schauder basis \seq{t}{n} which is
     conditional and dominates the summing basis of $c_0$. We
     shall discuss in more detail the above two ingredients in the
     rest of the introduction.

  In section \ref{Nsec1} we define a mixed Tsirelson space
   $T_0=T[(\mc{A}_{n_j},\frac{1}{m_j})_{j=1}^{\infty}]$
   with an unconditional basis, such that its dual space $T_0^*$
   admits
   a $c_0$ spreading model.  The space $T_0$ will be the unconditional
   frame required for the definition of the HI space $\eqs_d$, in a similar manner
    as
   Schlumprecht's space \cite{Sc} is the unconditional frame for the space
   $X_{GM}$ of Gowers and Maurey \cite{GM1} and as the asymptotic $\ell_1$ space $X_{ad}$ having an
   unconditional basis is the unconditional frame for the asymptotic $\ell_1$ HI space $X_{AD}$ \cite{AD}.
   The sequence \seq{m}{j} we use for the space $T_0$, as well as
   for the space  $\eqs_d$,  is inspired by Gasparis work
   (\cite{Ga}) and is defined recursively as follows
   \[  m_1=m_2=2,    \qquad \mbox{ and}\qquad m_{j}=m_{j-1}^2=m_1\cdot m_2\cdot\ldots\cdot m_{j-1}
      \quad\mbox{ for } j\ge 3, \]
    while we require that the sequence \seq{n}{j}  increases rather fast, namely
   \[ n_1\ge 2^3m_3  \qquad \mbox{ and}\qquad n_j\ge (4n_{j-1})^5\cdot m_j  \quad\mbox{ for }    j\ge 2.\]

   As it is well known, the norm of the space $T_0=T[(\mc{A}_{n_j},\frac{1}{m_j})_{j=1}^{\infty}]$
   satisfies the implicit formula
   \[  \|x\|=\max\{\|x\|_{\infty},\;\sup\limits_j \|x\|_j|\}  \]
   where $\|x\|_j=\frac{1}{m_j}\sup
   \sum\limits_{k=1}^{n_j}\|E_kx\|$ with the supremum taken over
   all families $(E_k)_{k=1}^{n_j}$ of successive finite sets.
   Note that the Schauder basis \seq{e}{l} of $T_0$ is
   subsymmetric and also each $\|\;\|_j$ is an equivalent norm on
   $T_0$.

   The fundamental property of mixed Tsirelson spaces, like the above
   $T_0$,
   is a biorthogonality described as follows. There exists a
   null sequence $(\e_i)_i$
   of positive numbers, such that
   for every  infinite dimensional subspace $Z$ and every
   $j\in\N$, there exists a vector $z\in Z$ with $\|z_j\|=\|z_j\|_j$
   and $\|z_j\|_i\le \e_{\min\{i,j\}}$.
   A transparent example of this phenomenon are the vectors of the
   form
   $y_j=\frac{m_j}{n_j}\sum\limits_{l=1}^{n_j}e_l$ in $T_0$,
   satisfying  the following properties. $\|y_j\|=\|y_j\|_j=1$ while
   $\|y_j\|_i\le \frac{2}{m_i}$ for $i<j$ and
   $\|y_j\|_i\le \frac{m_j}{m_i}$ for $i>j$.

    As follows from Gasparis method the above unique evaluation
    of the vectors $(y_j)_j$  is no longer true for all averages
    of the basis. More precisely setting
    $p_j=n_1\cdot n_2\cdot\ldots \cdot n_{j-1}$ the following
    holds.
    \begin{proposition}\label{Nprop13}
    For every $j\ge 3$ we have that
    \[\|\frac{1}{p_j}\sum\limits_{l=1}^{p_j}e_l\|\le\frac{4}{m_j}.\]
    \end{proposition}
    As $p_j=\prod\limits_{i=1}^{j-1}n_i$ and
    $m_j=\prod\limits_{i=1}^{j-1}m_i$,
    it is easily shown that
     $\frac{1}{m_j}\le \|\frac{1}{p_j}\sum\limits_{l=1}^{p_j}e_l\|_i$
     for $1\le i\le j$. Hence we conclude that, unlikely for the
     vectors $\frac{m_j}{n_j}\sum\limits_{l=1}^{n_j}e_l$,
     the vectors $\frac{m_j}{p_j}\sum\limits_{l=1}^{p_j}e_l$ have
     simultaneous evaluations by the family of norms $(\|\;\|_i)_{1\le i\le j}$.
      This actually yields that if we consider
     successive functionals $(\phi_j)_{j=3}^{\infty}$
     of the form $\phi_j=\frac{1}{m_j}\sum\limits_{l\in F_j}e_l^*$
     with $\#(F_j)=p_j$, then the sequence $(\phi_j)_{j=3}^{\infty}$
     generates a $c_0$ spreading model in $T_0^*$.
     The proof of Proposition \ref{Nprop13} is more involved than
     the corresponding one for the vectors
     $\frac{1}{n_j}\sum\limits_{l=1}^{n_j}e_l$ and requires some
     new techniques which could be of independent interest.

    The existence of a sequence generating a $c_0$ spreading model
    in the dual space $T_0^*$ is the basic tool for constructing
    strictly singular non-compact operators on $T_0$. This follows
    from the next general statement which is presented in
    Proposition \ref{Nprop12} of   section \ref{Nsec88}.
    \begin{proposition}
  Let $X,Y$ be a pair of Banach spaces such that
  \begin{enumerate}
  \item[(i)] There exists a sequence $(x_n^*)_{n\in\N}$ in $X^*$
  generating a $c_0$ spreading model.
  \item[(ii)] The space $Y$ has a normalized Schauder basis
  \seq{e}{n} and there exists a norming set $D$ of $Y$ (i.e.
  $D\subset Y^*$ and $\|y\|=\sup\{f(y):\;f\in D\}$ for every $y\in
  Y$), such that for every $\e>0$ there exists $M_{\e}\in\N$ such
  that for every $f\in D$, \[\#\{n\in\N:\; |f(e_n)|>\e\}\le
  M_{\e}.\]
  \end{enumerate}
  Then there exists a strictly increasing sequence of integers
 \seq{q}{n} such that the operator $T:X\to Y$ defined by the rule
 \[T(x)=\sum\limits_{n=1}^{\infty}x^*_{q_n}(x)e_n\] is
 bounded and non-compact.
 \end{proposition}

 The fact that every mixed Tsirelson space of the form
 $T[(A_{n_j},\frac{1}{m_j})_j]$ satisfies condition (ii) of the
 above proposition, yields that there exist  strictly singular
 non-compact operators $S:T_0\to T_0$.

    In section \ref{Nsec2}, the space $\eqs_d$ is defined with the use of the above defined sequences
    \seq{m}{j}, \seq{n}{j}. The norming set $K_d$ of the space $\eqs_d$ is defined to be the minimal
    subset of $c_{00}(\N)$ such that:
    \begin{enumerate}
    \item[(i)] It contains  $\{\pm e_n^*:\; n\in\N\}$.
    \item[(ii)] It is symmetric and closed under the restriction of its
    elements on intervals of $\N$.
    \item[(iii)] For each $j$, it is closed under the $(\mc{A}_{n_j},\frac{1}{m_j})$ operation.
    \item[(iv)] For each $j\ge 2$, it closed under the
    $(\mc{A}_{n_{2j-1}},\frac{1}{\sqrt{m_{2j-1}}})$ operation on $n_{2j-1}$ special sequences.
    \end{enumerate}
    The special sequences are defined in the standard manner
    with the use of a Gowers-Maurey type coding function $\sigma$.
    Notice that, since $m_{j+1}=m_j^2$ for $j\neq 1$, condition (iv) is equivalent
    to saying that the set $K_d$ is closed under the
    $(\mc{A}_{n_{2j+1}},\frac{1}{m_{2j}})$ operation on $n_{2j+1}$ special sequences
    for each $j$.
    Using the standard methods for this purpose, we prove that the space $\eqs_d$ is  HI.

    In section \ref{Nsec3}, a class of bounded diagonal operators on the space $\eqs_d$ is defined.
    These diagonal operators are of the form $\sum\limits_k \lambda_k D_{j_k}$, where $(D_{j_k})_k$ is a
    sequence of diagonal operators with successive finite dimensional ranges.  To be more
    precise,  for each $j$ and every choice of successive intervals
    $(I_i^j)_{i=1}^{p_j}$ we
     define a diagonal operator $D_j:\eqs_d\to \eqs_d$, by the rule
    \[  D_j(x)  =\frac{1}{m_j}\sum\limits_{i=1}^{p_j}I^j_ix.  \]
    Under certain growth conditions on the set $\{j_k:\; k\in\N\}$, we prove that for
    every $(\lambda_k)_{k\in\N}\in\ell_{\infty}(\N)$ the diagonal operator
    $D=\sum\limits_k \lambda_k D_{j_k}:\eqs_d\to\eqs_d$ is bounded with
    $\|D\|\le C_0 \cdot \sup\limits_k |\lambda_k|$  for some universal constant $C_0$.
    It easily follows that such an operator $D$ is strictly singular, since the
    space $\eqs_d$ is HI and $\lim\limits_n D(e_n)=0$ (Proposition 1.2 of \cite{AT1}).

    In order to construct strictly singular non-compact diagonal operators on $\eqs_d$
    we  prove that for appropriate choice of the intervals
    $\big((I^{j_k}_i)_{i=1}^{p_{j_k}}\big)_{k=1}^{\infty}$ the corresponding diagonal operator
    $\sum\limits_k D_{j_k}$ is non-compact.
    The main tool for studying the structure of the space of diagonal operators on
    $\eqs_d$,
     is the finite block representability
    of  $J_{T_0}$ in every block subspace of $\eqs_d$.
    The space $J_{T_0}$ is the Jamesification
    of the space $T_0$ described earlier.
    This class of spaces was defined
     by S. Bellenot, R. Haydon and E. Odell in \cite{BHO}.
    Using the language of mixed Tsirelson spaces, we may write
    \[J_{T_0}=T\big[G,\big(\mc{A}_{n_j},\frac{1}{m_j}\big)_{n\in\N}\big],  \]
     with $G=\{\pm \chi_I:\;I\text{ finite interval of }\N\}$.
     We prove that for every $N\in\N$ and every block subspace $Z$ of $\eqs_d$,
     there exists  a block sequence $(z_k)_{k=1}^N$ in $Z$ such that
      \begin{equation}\label{eq14}
      \|\sum\limits_{k=1}^N  \mu_kt_k\|_{J_{T_0}}\le
  \|\sum\limits_{k=1}^N  \mu_kz_k\|_{\eqs_d}\le c\cdot
   \|\sum\limits_{k=1}^N  \mu_kt_k\|_{J_{T_0}}
   \end{equation}
     for $c$ a  universal constant. The notation \seq{t}{n} stands for the standard basis
     of $J_{T_0}$. A similar result in a different context, is
     given by S. Argyros, J. Lopez-Abad and S. Todorcevic
      in    \cite{ALT1}, \cite{ALT2}. The precise definition of
       the space $J_{T_0}$ is given in section \ref{Nsec4},
      where the theorem of the finite block representability of $J_{T_0}$ in
      every block subspace of $\eqs_d$ is stated, postponing its proof for section
      \ref{Nsec5}.
   Section \ref{Nsec4} is mainly devoted to the construction of
   the diagonal strictly singular non-compact operators on the
   space $\eqs_d$.

  For a given block subspace $Z$ of $\eqs_d$, using \eqref{eq14}
  in conjunction with some easy estimates on the basis of $J_{T_0}$,
   we construct successive block sequences   $(y^j_k)_{k=1}^{2p_j}$ in $Z$,
  such that
   \begin{equation}\label{eq15}
  \|\frac{1}{2p_j}\sum\limits_{k=1}^{p_j}y^j_{2k-1}\|\ge\frac{1}{2}
  \qquad\text{ and }\qquad
  \|\frac{1}{2p_j}\sum\limits_{k=1}^{2p_j}(-1)^{k+1}y^j_{k}\|\le
  \frac{4c}{m_j}.
  \end{equation}
   We set $D_j(x)=\frac{1}{m_j}\sum\limits_{i=1}^{p_j}I_i^j x$ for each $j$,
   where $I^j_i=\ran(y^j_{2i-1})$.
   Let' s point out that the diagonal operator $D_j$ acting on the
   vector  $x_j=\frac{m_j}{2p_j}\sum\limits_{i=1}^{2p_j}(-1)^{i+1}y^j_i$,
   ignores $y^j_i$ when $i$ is even. This in conjunction with  \eqref{eq15} yields that
    $\|x_j\|\le 4c$, $\|D_jx_j\|\ge\frac{1}{2}$.
   For a suitable choice of the set $\{j_k:\; k\in\N\}$, the diagonal operator
   $D=\sum\limits_k D_{j_k}$ is bounded and strictly singular, while it is
   non-compact (even the restriction of $D$ on the subspace $Z$ is non-compact) since for the block
   sequence $(x_{j_k})_{k\in\N}$ we have that $\|x_{j_k}\|\le  4c$ while
    $\|Dx_{j_k}\|\ge \frac{1}{2}$.

     Moreover, it is easily shown that for every
    $(\lambda_k)_{k\in\N}\in\ell_{\infty}(\N)$,
   \[  \frac{1}{8c}\cdot \sup\limits_k|\lambda_k|\le
    \|\sum\limits_{k=1}^{\infty}\lambda_kD_{j_k}\|\le
 C_0\cdot\sup\limits_k|\lambda_k| \] hence the space $\mc{L}_{\diag}(\eqs_d)$ of
 diagonal operators of $\eqs_d$ contains an isomorphic copy of $\ell_{\infty}(\N)$.
      The next theorem summarizes the basic properties of the space $\eqs_d$.
      \begin{theorem}
      There exists a Banach space $\eqs_d$ with a Schauder basis \seq{e}{n} such that:
      \begin{enumerate}
      \item[(i)] The space $\eqs_d$ is reflexive and HI.
      \item[(ii)] For every infinite dimensional subspace $Z$ of $\eqs_d$ there exists a
      diagonal strictly singular operator $D:\eqs_d\to\eqs_d$ such that the restriction of
      $D$ on the subspace $Z$ is a non-compact operator.
      \item[(iii)] The space
      $\mc{L}_{\diag}(\eqs_d)$ of
  diagonal operators of $\eqs_d$ with respect to the basis \seq{e}{n}
  contains an isomorphic copy of $\ell_{\infty}(\N)$.
      \end{enumerate}
      \end{theorem}

   As we have mentioned above the scalar plus compact problem
   remains open within the class of separable reflexive Banach
   spaces. Even the weaker problem related to the present work,
   namely the existence of a reflexive  Banach  space with a Schauder basis
   such that every diagonal operator is of the form $\lambda I+K$,
   with $K$ a compact diagonal operator, is still open.
   In a forthcoming paper \cite{ADT}, we shall present a
   quasireflexive Banach space   $\eqs_D$  with a Schauder basis, such that
   the space $\mc{L}_{\diag}(\eqs_D)$ is HI and satisfies the
   scalar plus compact property.

\section{The mixed Tsirelson space $T_0= T[(\mc{A}_{n_j},\frac{1}{m_j})_{j=1}^{\infty
}]$}\label{Nsec1}

 This section
 is devoted to the construction of a mixed Tsirelson space $T_0$ with
 an unconditional basis, such that the dual space $T_0^*$ admits a sequence
 which generates a $c_0$ spreading model.
 This space is of the form $T_0=T[(\mc{A}_{n_j},\frac{1}{m_j})_{j\in\N}]$ with a very careful choice
 of the sequence $(m_j)_{j\in\N}$.

 \begin{notation}\label{Wnot1}
  For a finite  set $F$, we denote   by
 $\#F$   the cardinality of the set $F$. We
 denote by $\mc{A}_n$ the class of subsets of $\N$ with cardinality
 less than or equal to $n$,
\[\mathcal {A}_n=\{ F\subset \N:\; \# F\le n\}.\]

 By $c_{00}(\N)$ we denote the vector space of all finitely supported
 sequences of reals and  by either $(e_i)_{i=1}^{\infty}$ or
 $(e^*_i)_{i=1}^{\infty}$, depending on the context, its standard
 Hamel basis.
 For $x=\sum\limits_{i=1}^{\infty}a_ie_i \in c_{00}(\N)$, the support
  of $x$ is the set $\supp x=\{ i\in\mathbb{N}:\; a_i\neq
 0\}$ while the  range $\ran x$ of $x$,  is the smallest interval
 of $\mathbb{N}$ containing $\supp x$. For nonempty finite subsets
 $E, F$ of $\mathbb{N}$,  we write $E<F$ if $\max E < \min F$.
  For $n\in\mathbb{N}$,
 $E\subset \mathbb{N}$ we write $n<E$ (resp. $n\le E$) if $n<\min E$
 (resp. $n\le \min E$).
  For $x,y$ nonzero vectors  in $c_{00}(\N)$,  $x<y$ means
  $\supp x<\supp y$.
  For $n\in\mathbb{N}$, $x\in c_{00}(\N)$, we write
 $n<x$ (resp. $n\le x$) if $n< \supp x$ (resp. $n\le  \supp x$). We
 shall  call the subsets $(E_i)_{i=1}^n$ of $\N$   successive if
 $E_1<E_2<\cdots <E_n$.
  Similarly, the vectors
 $(x_i)_{i=1}^n$ are called  successive, if $x_1<x_2<\cdots < x_n$.
   For $x=\sum\limits_{i=1}^{\infty}a_ie_i$ and $E$ a  subset of
 ${\mathbb{N}},$ we  denote by $Ex$ the vector $Ex=\sum\limits_{i\in
 E}a_ie_i.$ Finally, for $f=\sum\limits_{i=1}^{\infty}\beta_ie_i^* \in c_{00}(\N)$
 and $x=\sum\limits_{i=1}^{\infty}a_ie_i \in c_{00}(\N)$
 we denote by $f(x)$  the real number
 $f(x)=\sum\limits_{i=1}^{\infty} a_i \beta_i$.
 \end{notation}

 \begin{definition}\label{Ndef1}
 Let $n\in {\mathbb N}$ and $\theta \in (0,1)$.
 \begin{enumerate}
 \item[(i)]  A finite sequence $(f_i)_{i=1}^k$ in $c_{00}(\N)$ is
 said to be $\mc{A}_n$ admissible if $k\le n$ and $f_1<f_2<\cdots
 <f_k$.
 \item[(ii)] The $(\mathcal{A}_n,\theta)$ operation on $c_{00}(\N)$
 is the operation which
 assigns to each $\mc{A}_n$  admissible sequence
 $f_1<f_2<\cdots <f_k$ the vector $\theta (f_1+f_2+\cdots
 +f_k)$.
 \end{enumerate}
 \end{definition}

  \begin{definition}\label{Ndef2}
 Given a pair $(m_j)_{j\in I}$, $(n_j)_{j\in I}$ of either finite ($I=\{1,\ldots,k\}$)
 or infinite ($I=\N$) increasing sequences of
 integers we  shall denote by
 $K=K[(m_j,n_j)_{j\in I}] $ the minimal subset of $c_{00}(\N)$ satisfying
 the following conditions.
 \begin{enumerate}
 \item[(i)] $\{\pm e_i^*:\;i\in {\mathbb N}\}\subset K$.
 \item[(ii)] For each $j\in I$, $K$ is closed under the
 $(\mc{A}_{n_j},\frac{1}{m_j})$ operation.
 \end{enumerate}
 \end{definition}

 It is easy to check that the set  $K$ is symmetric and closed
 under the restriction of its elements on subsets of $\N$.

 Let $j\in \N$. If  $f\in K$  is the result of the
 $(\mc{A}_{n_j},\frac{1}{m_j})$ operation on some sequence
 $f_1<f_2<\cdots <f_k$ $(k\le n_j)$ in $K$, we shall say that the
 weight of $f$ is $m_j$ and we shall  denote this fact  by
 $w(f)=m_j$. We note however that the weight $w(f)$ of a functional $f\in K$ is not
 necessarily uniquely determined.

 \begin{definition}\label{Ndef3}[The tree ${T_f}$
 of a functional ${ f\in K}$] \label{tree} Let $f\in K$. By a
 tree  of $f$ (or tree corresponding to the analysis of $f$)
 we mean a finite family
 $T_f=(f_a)_{a\in\mc{A}}$ indexed by a finite tree $\mc{A}$ with a
 unique root $0\in\mc{A}$ such that the following conditions are
  satisfied:
 \begin{enumerate}
 \item[(i)] $f_0=f$ and $f_a\in K$ for all $a\in \mc{A}$.
\item[(ii)] If $a$ is maximal in $\mc{A}$,  then $f_a=\pm e_k^*$
for some $k\in \N$.
 \item[(iii)] For every $a\in \mc{A}$ which is not maximal  denoting by
 $S_a$ the set of immediate successors of $a$ in $\mc{A}$ the following holds.
 There exists $j\in{\mathbb N}$  such that the family $(f_\beta)_{\beta\in
 S_a}$  is $\mc{A}_{n_j}$  admissible and
 $f_a=\frac{1}{m_j}\sum\limits_{\beta\in S_a}f_\beta$. In this case we say that
 $w(f_a)=m_j$.
 \end{enumerate}
 The order $o(f_a)$ for each $a\in\mc{A}$
 is also defined by backward induction as follows. If $f_a=\pm e_k^*$ then $o(f_a)=1$, while if
 $f_a=\frac{1}{m_j}\sum\limits_{\beta\in S_a}f_\beta$ then
 $o(f_a)=1+\max\{o(f_{\beta}:\;\beta\in S_a\}$.
 The order $o(T_f)$ of the aforementioned tree is defined to be equal to $o(f_0)$
  (where $0\in\mc{A}$ is the unique root
 of the tree $\mc{A}$).
  \end{definition}

 \begin{remark}\label{Nrem1}
 An easy inductive argument yields the following.
 \begin{enumerate}
 \item[(i)]   Every $f\in K$ admits a  tree, not necessarily unique.
  \item[(ii)] For every $\phi\in K$, if
 $\supp(\phi)=\{k_1<k_2<\cdots<k_d\}$ then for every
 $l_1<l_2<\cdots<l_d$ in $\N$ the functional
 $\psi=\sum\limits_{i=1}^d \phi(e_{k_i})e_{l_i}^*$ also belongs to
 the set $K$.
 \item[(iii)] For every $\phi\in K$ and every $E\subset \N$ the
 functional $E\phi$ also belongs  to the  set $K$.
 \item[(iv)] If  $\phi=\sum\limits_{i=1}^{\infty}a_ie_i\in K$,
  then for every choice of
 signs $(\e_i)_{i=1}^{\infty}$ the functional
 $\sum\limits_{i=1}^{\infty}\e_ia_ie_i$ also  belongs to $K$.
 \end{enumerate}
 \end{remark}

 \begin{definition}\label{Ndef26}
  The order
 $o(f)$ of an $f\in K$, is defined as
  \[ o(f)=\min\{o(T_f):\; T_f \mbox{ is a tree of }f\}.\]
 \end{definition}

 In general, given a symmetric subset $W$ of $c_{00}(\N)$ containing
 $\{\pm e_k^*:\;k\in\N\}$, the norm induced by $W$ on
 $c_{00}(\N)$ is defined as follows. For every $x\in c_{00}(\N)$,
 \[ \| x\|_W=\sup \{f(x):\; f\in W\}. \]
 In the case where $W=K=K[(m_j,n_j)_{j\in I}]$ for a given
 double sequence $(m_j,n_j)_{j\in I}$, the completion of
 the corresponding normed space $(c_{00}(\N),\|\cdot \|_K)$ is denoted
 by $ T[ ( \mc{A}_{n_j}, \frac{1}{m_j})_{j\in I}]$ and is called
 the mixed Tsirelson space defined by the family
 $(\mc{A}_{n_j},\frac{1}{m_j})_{j\in I}.$ The norming set $K$ is called the
  standard norming set of the space
 $ T[ ( \mc{A}_{n_j}, \frac{1}{m_j})_{j\in I}]$.

 \begin{remark}\label{Nrem2}
 \begin{enumerate}
 \item[(i)] As follows from Remark \ref{Nrem1}(iii), (iv),
  the Hamel basis \seq{e}{i} of
 $c_{00}(\N)$ is  a 1-unconditional Schauder basis for the
 space $ T[(\mc{A}_{n_j}, \frac{1}{m_j})_{j\in I}]$.
 \item[(ii)] If $x\in c_{00}(\N)$ with $\supp
 x=\{k_1<k_2<\cdots<k_d\}$ and $l_1<l_2<\cdots<l_d$ then the
 vector $y=\sum\limits_{i=1}^d e_{k_i}^*(x)e_{l_i}$ satisfies
 $\|x\|_K=\|y\|_K$, thus the basis \seq{e}{i} is subsymmetric.
 This is also a consequence of Remark
 \ref{Nrem1}.
 \end{enumerate}
 \end{remark}

 For the definition of the space $T_0$ and of the Hereditarily Indecomposable space
 $\eqs_d$  later,
  we shall use a specific choice of the sequences
 $(m_j)_{j\in\N}$ ,$(n_j)_{j\in\N}$ described in the next  definition.
 In the sequel  $(m_j)_{j\in\N}$, $(n_j)_{j\in\N}$ will always
 stand for  these sequences.

 \begin{definition}\label{Ndef4}[The sequences $(m_j)_{j\in\N}$ ,$(n_j)_{j\in\N}$ and the
 space $T_0$]\\
 We set $m_1=m_2=2$, and for $j\geq 3$ we define
 \[m_j=m_{j-1}^2=\prod\limits_{i=1}^{j-1}m_i.\]
  We choose a sequence $(n_j)_{j=1}^{\infty }$ as follows: $n_1\geq 2^3m_3$,
 and for every $j\geq 2$ we
 choose
 \[n_j\ge (4n_{j-1})^5\cdot m_j\]
 Observe, for later use, that $n_j\ge 2^{j+2}m_{j+2}$ while,
 setting  $p_j=n_1\cdot n_2\cdot
\ldots \cdot n_{j-1}$, we have that $n_j\ge jp_j$. We notice here
that the numbers $(p_j)_{j\ge 3}$ will play a key role in our
proofs.

 We set
\[T_0=T\left [\left ( {\mathcal A}_{n_j},\frac{1}{m_j}\right
)_{j=1}^{\infty }\right ]\] and we denote by $K_0$ the standard
norming set of $T_0$.
 \end{definition}

 Our aim is to prove that $T_0^{\ast }$ has a block sequence which
 generates a $c_0$ spreading model (Proposition \ref{Nprop1}). The main step of
 the proof is done in Lemma \ref{Nlem2}. For its proof we need to recall
 the definition of the modified Tsirelson spaces $T_M [ ({\mathcal
 A}_n,\theta_n)_{n\in I}]$. For a given (finite or infinite) subset
 $I$ of ${\mathbb N}$ and a sequence $(\theta_n)_{n\in I}$ in
 $(0,1)$, with $\lim\limits_{n\in I, n\rightarrow\infty }\theta_n=0$
 if $I$ is infinite,  the set $K_M=K_M[ ({\mathcal
 A}_n,\theta_n)_{n\in I}]$ is defined as follows:

 The set $K_M$ is the
 minimal subset of $c_{00}(\N)$ with the following properties:

\begin{enumerate}
\item[(i)] $\{\pm e_k^{* }    :k\in {\mathbb N}\}\subset K_M$.
\item[(ii)] For every $n\in I$, every $m\le n$ and every sequence
 $(\phi_k)_{k=1}^m$ in $K_M$ with pairwise disjoint supports,
 we have that
$\theta_n\big (\sum\limits_{k=1}^m\phi_k\big )\in K_M$.
\end{enumerate}

 We define the norm $\|\cdot\|_M$ on $c_{00}(\N)$ by the rule
 \[\| x\|_M=\sup\{  \phi(x) :\;\phi\in K_M\}\] for every $x\in c_{00}(\N)$.
 The space $T_M[ ({\mathcal A}_n,\theta_n)_{n\in I} ]$ is the
 completion of the space $(c_{00}(\N),\|\cdot\|_M)$.

 It is proved in \cite{BD} that a space of
 the form $X=T[({\mathcal A}_{n_i},\frac{1}{m_i})_{i=1}^{k}]$ is
 isomorphic to $\ell_p(\N)$ for some $1<p<\infty $ (or $c_0(\N)$). Under
 the condition that the sequence $\big ( \log_{m_i}(n_i)\big
 )_{i=1}^k$ is increasing (which is satisfied by the sequences
 $(m_i)$ and $(n_i)$ used in the definition of $T_0$) this $p$ is
 the conjugate exponent of $q=\log_{m_k}(n_k)$. In particular, it
 is shown in \cite{BD} that, for every $f\in c_{00}(\N)$, we have
 $\|f\|_q\le\|f\|_{X^*}$, where $\|\cdot\|_q$ denotes the
 norm of $\ell_q(\N)$.

Using the same argument (induction and H\"{o}lder's inequality)
one can also get the inequality $\|f\|_q\le\|f\|_{X_M^*}$ where
$\|\cdot\|_{X_M^*}$ is the norm of the dual of the modified space
$X_M=T_M[({\mathcal A}_{n_i},\frac{1}{m_i})_{i=1}^{k}]$. We note
for completeness that, using the obvious inequality
$\|f\|_{X_M^*}\le\|f\|_{X^*}$, we get that in fact $X_M$ is
isomorphic to $X$ (and $\ell_p(\N)$).

 \begin{lemma}\label{Nlem1}
 Let $j\in {\mathbb N}$, $j\geq 3$. We denote by $K_M(j-2)$ the
 norming set of the modified space
 $T_M[(\mc{A}_{n_i},\frac{1}{m_i})_{i=1}^{j-2}]$.
 Let $\phi\in
 K_M(j-2)$ be such that, for every $l\in \supp(\phi )$,
 we have that $\phi(e_l)>\frac{1}{m_j}$. Then,
 \[\#\supp(\phi )\le n_{j-1}.\]
 \end{lemma}
 \begin{proof}[\bf Proof.] For the space
 $X_M=T_M[({\mathcal A}_{n_i},\frac{1}{m_i})_{i=1}^{j-2}]$
 where $(m_i)_i$ and $(n_i)_i$ are as in the definition of $T_0$, the
 inequality $\|\phi\|_q\le\|\phi\|_{X_M^*}$ with
 $q=\log_{m_{j-2}}(n_{j-2})$ implies the following: If $\phi\in
 B_{X_M^{\ast }}$ and $\phi (e_l)>\frac{1}{m_j}$
 for every $l\in\supp(\phi)$, then
 \[\frac{(\#\supp(\phi ))^{1/q}}{m_j}<\|\phi \|_q\le \|\phi\|_{X_M^*}\le 1.\]
 Since $m_j=m_{j-2}^4$ and $n_{j-2}^4<n_{j-1}$, we
  get that $\#\supp(\phi )<m_j^q=(m_{j-2}^q)^4=n_{j-2}^4\le n_{j-1}$.
 \end{proof}

 \begin{lemma} \label{Nlem2} Let $j\ge 3$
 and let $k_1<k_2<\cdots<k_{p_j}$. Then
 \[ \|\frac{1}{p_j}\sum\limits_{i=1}^{p_j}e_{k_i}\|\le
 \frac{4}{m_j} .\]
 \end{lemma}
 \begin{proof}[\bf Proof.]
 From the subsymmetricity of the basis \seq{e}{i} (Remark \ref{Nrem2}(ii))  it is enough to show that
 $\|\sum\limits_{l=1}^{p_j}e_l\|\le \frac{4 p_j}{m_j}$.
 Let $f\in K_0$; we shall show that
 $f(\sum\limits_{l=1}^{p_j}e_{l})\le \frac{4p_j}{m_j}$.
   We may assume that $f(e_l)\ge 0$ for all $l$
 (Remark \ref{Nrem1}(iv)).

  We set $D=\{l\in\supp(f):\;
 \phi(e_l)>\frac{1}{m_j}\}$ and we define
 $\phi =f|_D$ and $\psi =f|_{\N\setminus D}$.
 Since obviously $\psi(\sum\limits_{i=1}^{p_j}e_{i})\le
 \frac{p_j}{m_j}$ it is enough to show that
 $\phi(\sum\limits_{i=1}^{p_j}e_{i})\le
 \frac{3p_j}{m_j}$.

 Fix a tree analysis $T_{\phi }=(\phi_a)_{a\in\mc{A}}$ of the functional
 $\phi$. For
 every $l\in \supp(\phi )$
 we define the set
 $A_{l }=\{i:\;\exists a\in\mc{A}\mbox{ with }l\in\supp(f_a)\mbox{ and }w(f_a)=m_i\}$
 and for each $i\in A_l$ we denote by
  $d_{l,i}$ the cardinality of the set $\{a\in\mc{A}:\;l\in\supp(f_a)\mbox{ and }w(f_a)=m_i\}$.
  Then, for each $l\in \supp(\phi )$,
 \[\prod_{i\in
 A_l}\frac{1}{m_i^{d_{l,i}}}=\phi (e_l)>\frac{1}{m_j}.\]
 Thus we have that $\prod\limits_{i\in A_l}m_i^{d_{l,i}}<m_j$ which in
 conjunction to the fact that $d_{l,i}\ge 1$ for each $i\in A_l$
 and taking into account that $m_j=m_1\cdot m_2\cdot \ldots \cdot
 m_{j-1}$ we get the following:
 \begin{enumerate}
 \item[(1)] $A_l$ is a proper subset of $\{ 1,\ldots ,j-1\}$.
 \item[(2)] If $j-1\in A_l$, then $d_{l,j-1}=1$. In general, if
 $j-1,\ldots ,j-k\in A_l$, then $d_{l,j-1}=d_{l,j-2}=\cdots
 =d_{l,j-k}=1$.
 \end{enumerate}

 We partition the set $\supp(\phi )$ in the sets
 $(B_i)_{i=1}^{j-1}$ defined as follows. We set
  \[ B_{j-1}=\{l\in \supp(\phi ):\; j-1\notin A_l\},\]
 and for $k=2,\ldots ,j-1$, we set
 \[B_{j-k}=\{ l\in\supp(\phi ):\; j-i\in A_l
 \mbox{ for }1\le i<k\mbox{ and }j-k\not\in A_l\}.\]

  In the following
 three steps we estimate the action of $\phi $ on $B_{j-1}$,
 $B_{j-2}$ and (in the general case) on $B_{j-k}$.

\medskip

 \noindent {\bf Step 1}. The functional $\phi |_{B_{j-1}}$
 satisfies the assumptions of Lemma \ref{Nlem1}, hence
  \[|\phi(\sum\limits_{l\in B_{j-1}}e_i)|\le
     \#(\supp\big(\phi|_{B_{j-1}})\big)\le  n_{j-1}.\]
 \noindent {\bf Step 2.} Let $\phi'=\phi |_{B_{j-2}}$ and
 let $T_{\phi'}=(f_a)_{a\in \mc{A'}}$ be the restriction of the
 analysis $T_{\phi }$ on $B_{j-2}$. Then, for every $l\in
 \supp(\phi')=B_{j-2}$, there exists exactly one $a\in
 \mc{A'}$ such that $l\in \supp(f_a)$ and
 $w(f_a)=m_{j-1}$.

 \begin{claim} There exist disjointly supported functionals
 $(\phi_s)_{s=1}^{n_{j-1}}$ such that
 $$\phi^{\prime
 }=\frac{1}{m_{j-1}}\sum_{s=1}^{n_{j-1}}\phi_s$$ with
 $\phi_s\in K_M(j-3)$ for $1\le s\le n_{j-1}$.
 \end{claim}
 \begin{proof}[\bf  Proof of the Claim.]
  Let \[\mc{B}=\{a\in \mc{A'}:\; w(f_{a
 })=m_{j-1}\}.\] By the definition of $\phi^{\prime }$, the
 functionals $(f_{a })_{a\in\mc{B}}$, have pairwise  disjoint
 supports and
 \[\bigcup_{a\in {\mathcal B}}\supp(f_{a})=\supp(\phi^{\prime }).\]
  For each $a\in \mc{A'}$ we
 write
 \[f_{a}=\frac{1}{m_{j-1}}\sum_{\beta\in S_{a }}f_{\beta
 }=\frac{1}{m_{j-1}}\sum_{k=1}^{n_{j-1}}f_{\widehat{a k}},\]
 where $f_{\widehat{a k}}=0$ if $\#S_{a }<k\le n_{j-1}$.

We now build the disjointly supported functionals
$(\phi_s)_{s=1}^{n_{j-1}}$. We fix $s$ and we define inductively
the analysis $(f_{\gamma }^s)_{\gamma\in \mc{A'}}$ of $\phi_s$ as
follows: Let $\gamma\in \mc{A'}$ be a maximal node, i.e.
$f_{\gamma }=e^{* }_{k_{\gamma }}$. Then there exists a unique
$a\in \mc{B}$ such that $a \prec\gamma $. If $\gamma =\widehat{a
s}$ or $\widehat{a s}\prec\gamma $, then we set $f^s_{\gamma
}=f_{\gamma }$. Otherwise, we set $f^s_{\gamma }=0$.

Let now $\gamma\in \mc{A'}$, $\gamma $ not maximal, with
$f_{\gamma }=\frac{1}{m_r}\sum\limits_{\beta\in S_{\gamma
}}f_{\beta }$ and assume that $f^s_{\beta }$, $\beta\in S_{\gamma
}$, have been defined. If $\gamma\notin {\mathcal B}$ then we set
$f^s_{\gamma }=\frac{1}{m_r}\sum\limits_{\beta\in S_{\gamma
}}f^s_{\beta }$. If $\gamma\in {\mathcal B}$ then $f_{\gamma
}=\frac{1}{m_{j-1}}\sum\limits_{k=1}^{n_{j-1}}f_{\widehat{\gamma
k}}$ and we set $f^s_{\gamma }=f_{\widehat{\gamma s}}$.

This completes the inductive construction. It is now easy to check
that the functionals $\phi_s=f^s_0$, $s=1,\ldots ,n_{j-1}$,
(recall that $0\in\mc{A}$ is the unique root of the tree $\mc{A}$)
have the desired properties and this completes the proof of the
claim.
 \end{proof}
  Since $\phi'(e_l)=\phi(e_l)>\frac{1}{m_j}$ for each $l\in\supp(\phi')$
  it follows that for every $s$, $1\le s\le n_{j-1}$
   and every $l\in \supp(\phi_s)$, we have that
 \[\phi_s(e_l)>\frac{m_{j-1}}{m_j}=\frac{1}{m_{j-1}}.\]
 Thus, for every $s=1,\ldots ,n_{j-1}$, the functional $\phi_s$
 satisfies the assumptions of Lemma \ref{Nlem1}, with $j$ replaced by
 $j-1$, so
 \[\#\supp(\phi_s)\le n_{j-2}.\]
 It follows that
 \[\phi(\sum\limits_{l\in B_{j-2}}e_l)=\frac{1}{m_{j-1}}
 \sum\limits_{s=1}^{n_{j-1}}\phi_s(\sum\limits_{l\in
 \supp(\phi_s)}e_l)
\le\frac{1}{m_{j-1}}n_{j-1}n_{j-2}.\]

 \noindent {\bf Step 3.} Let $3\le k\le j-1$, set
 $\phi'=\phi|_{B_{j-k}}$, and let $T_{\phi'}=(f_a)_{a\in\mc{A'}}$ be the corresponding
analysis. Then, for every $l\in {\rm supp}(\phi^{\prime })$ and
for every $i=1,\ldots ,k-1$, there exists
 exactly one $a\in
 \mc{A'}$ such that $l\in \supp(f_a)$ and
 $w(f)=m_{j-i}$.
 As in Step 2, it follows by induction that we
can write
$$\phi^{\prime }=\frac{1}{m_{j-1}}\frac{1}{m_{j-2}}\cdots \frac{1}{m_{j-k+1}}
\left (\sum_{s=1}^{n_{j-k+1}\cdot\ldots\cdot n_{j-1}}\phi_s\right
),$$ where the functionals $(\phi_s)_{s=1}^{n_{j-k+1}\cdot\ldots
\cdot  n_{j-1}}$ have pairwise disjoint supports  for $1\le s\le
n_{j-k+1}\cdot\ldots \cdot  n_{j-1}$, $\phi_s\in K_M(j-k-1)$ while
for every $l\in \supp(\phi_s)$,
 \[\phi_s(e_l)>\frac{m_{j-1}\cdot\ldots\cdot
 m_{j-k+1}}{m_j}=\frac{1}{m_{j-k+1}}.\]
For every $s$, the functional $\phi_s$ satisfies the assumptions
of Lemma \ref{Nlem1} with $j$ replaced by $j-k+1$, so
 \[\#\supp(\phi_s)\le n_{j-k}.\]
It follows that
 \[\phi(\sum_{l\in B_{j-k}}e_l )\le
\frac{1}{m_{j-1}\cdot m_{j-2}\cdot\ldots\cdot m_{j-k+1}}\cdot
(n_{j-1}\cdot n_{j-2}\cdot\ldots\cdot n_{j-k+1})\cdot n_{j-k}.\]

\noindent We conclude that
\begin{eqnarray*}
 \phi (\sum_{l=1}^{p_j}e_l)&\le &
   \phi(\sum_{l\in B_{j-1}}e_l) +\cdots +
  \phi (\sum_{l\in B_1}e_l) \\
 &\le & n_{j-1}+\frac{1}{m_{j-1}}n_{j-1}n_{j-2}+\cdots +
 \frac{1}{m_{j-1}\cdot\ldots\cdot m_{j-k+1}}n_{j-1}\cdot\ldots\cdot n_{j-k}\\
 &&+\cdots+
 \frac{1}{m_{j-1}\cdot\ldots\cdot m_2}n_{j-1}\cdot\ldots\cdot n_1\\
 &=& n_{j-1}+\frac{1}{m_j}
 \big(\sum_{k=2}^{j-1}\frac{m_j}{m_{j-1}\cdot\ldots\cdot
   m_{j-k+1}}n_{j-1}\cdot\ldots\cdot n_{j-k}\big)\\
 &=& \frac{1}{m_j}\big(\sum_{k=1}^{j-1}m_{j-k+1}n_{j-1}\cdot\ldots\cdot n_{j-k}\big)\\
 && \hbox{(using the property}\; n_i\geq 2^{i+2}m_{i+2}\hbox{)}\\
 &\le & \frac{1}{m_j}
 \big(\sum_{k=1}^{j-2}\frac{1}{2^{j-k+1}}n_{j-k-1}
 n_{j-k}\cdots n_{j-1} \big)+\frac{m_2}{m_j}n_1\cdot\ldots\cdot n_{j-1}\\
 &\le & \frac{1}{m_j} \big(\sum_{k=1}^{j-2}\frac{1}{2^{j-k+1}}
  \big)
 p_j+\frac{2}{m_j}p_j\\
 &\le & \frac{3p_j}{m_j}.
 \end{eqnarray*}
 The proof of the lemma is complete.
 \end{proof}

 \begin{definition}\label{Ndef27}
 We say that a  sequence \seq{z}{n} in a  Banach space $Z$
 generates a $c_0$ spreading model provided that there exists a
 constant $C\ge 1$ such that for every $s\le k_1<k_2<\cdots<k_s$,
 the finite sequence $(z_{k_i})_{i=1}^s$ is $C$ equivalent to the
 standard basis of $\ell_{\infty}^n$.
 \end{definition}

 \begin{remark}\label{Nrem9}
 A sequence \seq{z}{n} generating a $c_0$ spreading model is
 necessarily weakly null. Indeed, assume the contrary. Then there
 exists $\e>0$, $f\in Z^*$ and $M$ an infinite sequence of $\N$
 such that $f(z_n)\ge \e$ for all $n\in M$. Choose
 $s>\frac{C}{\e}$ (where $C$ is the constant of the $c_0$
 spreading model) and $s\le k_1<k_2<\cdots<k_s$ with $k_i\in M$.
 Then from our assumption about the sequence \seq{z}{n} we get
 that $\|z_{k_1}+z_{k_2}+\cdots+z_{k_s}\|\le C$. On the other hand
 the action of the functional $f$ yields that
 $\|z_{k_1}+z_{k_2}+\cdots+z_{k_s}\|\ge
 \sum\limits_{i=1}^sf(z_{k_i})\ge s\e>C$, a contradiction.
 \end{remark}

 \begin{proposition} \label{Nprop1}
 There exists a block sequence in $T_0^*$
 which generates a $c_0$ spreading model.
 \end{proposition}
 \begin{proof}[\bf Proof.] Let $(F_j)_{j=3}^{\infty }$ be a
sequence of successive subsets of $\N$ with $\#F_j=p_j$, for each
$j=3,4,\ldots $. For $j=3,4,\ldots$ we set \[\phi_j=\frac{1}{m_j}\sum_{k\in
F_j}e_k^*.\]

 Then $\phi_j\in K_{0}$, thus $\|\phi_j\|\le 1$, and
 \[\phi_j(\frac{1}{p_j}\sum_{k\in F_j}e_k)=\frac{1}{m_j}.\]
  From Lemma \ref{Nlem2},
 we get that
 \[\|\frac{1}{p_j}\sum_{k\in F_j}e_k\|\le\frac{4}{m_j}.\]
  It follows that, for every
 $j=3,4,\ldots $
 \[\frac{1}{4}\le\|\phi_j\|\le 1.\]
  We shall show that the
sequence $(\phi_j)_{j=3}^{\infty }$ generates a $c_0$ spreading
model.
 This is a direct consequence of the following:
 \begin{claim}
  For every $s\in \N$, $s\geq 3$, and every choice of indices $j_1<j_2<\cdots <j_s$ with $s\le
 j_1$, the functional $\sum\limits_{k=1}^s\phi_{j_k}$ belongs to
 $K_{0}$.
 \end{claim}
  \begin{proof}[\bf Proof of the Claim.] Fix $s$ and $j_1<j_2<\cdots
<j_s\in \N$ with $s\le j_1$. For every $k=2,3,\ldots ,s$, we write
 \[\phi_{j_k}=\frac{1}{m_{j_k}}\sum_{i\in
 F_{j_k}}e_i^*= \frac{1}{m_{j_1}}\frac{1}{m_{j_1}\cdot
 m_{j_1+1}\cdot\ldots\cdot m_{j_k-1}}\sum_{i\in F_{j_k}}e_i^*.\]
   Since
 \[\#F_{j_k}=p_{j_k}=n_1\cdot\ldots\cdot n_{j_1-1}\cdot
 n_{j_1}\cdot\ldots\cdot n_{j_k-1}= p_{j_1}\cdot
 (n_{j_1}\cdot n_{j_1+1}\cdot\ldots\cdot n_{j_k-1}),\] we can
 partition the set  $F_{j_k}$ into $p_{j_1}$ successive subsets $(G_l^k)_{l=1}^{p_{j_1}}$
 where $\#G^k_l= n_{j_1}\cdot n_{j_1+1}\cdot\ldots\cdot n_{j_k-1}$ for every $l=1,\ldots ,p_{j_1}$.

 Then, for every $l=1,\ldots ,p_{j_1}$, the functional
 \[\psi_l^k=\frac{1}{m_{j_1}\cdot m_{j_1+1}\cdot\ldots\cdot m_{j_k-1}}\sum_{i\in
 G_l^k}e_i^{\ast }\] belongs to $K_{0}$. It follows that, for every
 $k=2,\ldots ,s$, we can write
 \[\phi_{j_k}=\frac{1}{m_{j_1}}\sum_{l=1}^{p_{j_1}}\psi_l^k\]
  where $\supp(\psi_1^k)
 < \supp(\psi_2^k) <\cdots < \supp(\psi_{p_{j_1}}^k)$ and
 $\psi_l^k\in K_{0}$ for every $l=1,\ldots ,p_{j_1}$. Since $s\le
 j_1$ and $sp_{j_1}\le j_1p_{j_1}\le n_{j_1}$, we get that the
 functional
 \[\phi =\sum_{k=1}^s\phi_{j_k}=\frac{1}{m_{j_1}}
 \left (\sum_{i\in F_{j_1}}e_i^{\ast }+\sum_{k=2}^s\left
 (\sum_{l=1}^{p_{j_1}}\psi_l^k\right )\right )\] belongs to
 $K_{0}$. This completes the proof of the Claim.
 \end{proof}
 The proof of the claim finishes, as we have mentioned earlier,
 the proof of the proposition.
 \end{proof}

\section{Strictly singular non-compact operators on $T_0$}\label{Nsec88}

 The main step in other examples, where strictly singular non-compact operators are
 produced on Hereditarily Indecomposable Banach spaces (e.g. \cite{AnSc}, \cite{Ga}),
 is the contsruction of strictly singular non-compact operators on the mixed Tsirelson
 spaces which are the unconditional frames of those space.
  In this section, we show how the existence of a sequence
 generating a $c_0$ spreading model
 in $T_0^*$ (Proposition \ref{Nprop1}),
 leads to strictly singular non-compact operators on $T_0$.
  In Proposition \ref{Nprop12}, which is of general nature,
  we prove how the existence of a $c_0$ spreading model in $X^*$
  leads to strictly singular non-compact operators $T:X\to Y$ for
  certain spaces $Y$, and then we apply this proposition to obtain the aforementioned
  result.

  We also notice that it is not known whether   each mixed Tsirelson space which is arbitrarily
 distortable admits a strictly singular non-compact operator.

 \begin{proposition}\label{Nprop12}
 Let $X,Y$ be a pair of Banach spaces such that
 \begin{enumerate}
 \item[(i)] There exists a sequence $(x_n^*)_{n\in\N}$ in $X^*$
 generating a $c_0$ spreading model.
 \item[(ii)] The space $Y$ has a normalized Schauder basis
 \seq{e}{n} and there exists a norming set $D$ of $Y$ (i.e.
 $D\subset Y^*$ and $\|y\|=\sup\{f(y):\;f\in D\}$ for every $y\in
 Y$), such that for every $\e>0$ there exists $M_{\e}\in\N$ such
 that for every $f\in D$, \[\#\{n\in\N:\; |f(e_n)|>\e\}\le
 M_{\e}.\]
 \end{enumerate}
 Then there exists a strictly increasing sequence of integers
 \seq{q}{n} such that the operator $T:X\to Y$ defined by the rule
 \[T(x)=\sum\limits_{n=1}^{\infty}x^*_{q_n}(x)e_n\] is
 bounded and non-compact.
 \end{proposition}
 \begin{proof}[\bf Proof.] Since the sequence $(x_n^*)_{n\in\N}$
 generates a $c_0$ spreading model  it is weakly null, hence, since it belongs to
 a dual space
  is  also $w^*$
 null. From a result of W. B. Johnson and H. P. Rosenthal
 (\cite{JR}), passing to a subsequence we may assume that
 $(x_n^*)_{n\in\N}$ is a $w^*-$ basic sequence. In particular
 there exists a bounded sequence $(x_n)_{n\in\N}$ in $X$ such that
 $(x_n,x_n^*)_{n\in\N}$ are biorthogonal (i.e.
 $x_i^*(x_j)=\delta_{ij}$ for each $i,j$).

 We select $\seq{\theta}{j}$ a strictly decreasing sequence of
 positive reals, with $\theta_1=1$, such that
 $\sum\limits_{j=1}^{\infty}j\theta_j<\infty$. From our assumption
 (ii) we may select a strictly increasing sequence \seq{q}{n} in
 $\N$ such that for every $j\in\N$ and every $f\in D$,
 \[\#\{n\in\N:\; |f(e_n)|>\theta_{j+1} \}\le q_j.\]
 We claim that the operator $T:X\to Y$ defined by the rule
 $T(x)=\sum\limits_{n=1}^{\infty}x^*_{q_n}(x)e_n$ is bounded and
 non-compact.

 We first show the boundedness of the operator $T$. Let $x\in X$
 and $f\in D$. For each $j$ we set
 \[B_j=\{n\in\N:\; \theta_{j+1}<|f(e_n)|\le \theta_j\}.\]
 From the definition of the sequence \seq{q}{j} it follows that
 $\#(B_j)\le q_j$. We partition each set $B_j$ in the following
 way:
 \[  C_j=\{n\in B_j:\;n\ge j\}\quad\mbox{ and }\quad D_j=\{n\in B_j:\;n< j\}\]
 Obviously $\#(D_j)\le j-1$. Since also
 \[ \#\{q_n:\;n\in C_j\}=\#(C_j)\le \#(B_j)\le q_j\le\min\{q_n:\;
 n\in C_j\},\]
 using our assumption (i), it follows that $\sum\limits_{n\in
 C_j}|x_{q_n}^*(x)|\le C\|x\|$, where $C$ is the constant of the
 $c_0$ spreading model.
  Thus for each $j$,
  \begin{eqnarray*}
  \sum\limits_{n\in B_j} |f(e_n)|\cdot |x^*_{q_n}(x)|&=&
  \sum\limits_{n\in C_j} |f(e_n)|\cdot |x^*_{q_n}(x)|+
  \sum\limits_{n\in D_j} |f(e_n)|\cdot |x^*_{q_n}(x)|\\
  & \le & \theta_j C\|x\|+ \theta_j(j-1) C\|x\|=j\theta_j C\|x\|.
  \end{eqnarray*}
 It follows that
  \begin{eqnarray*}
  |f(\sum\limits_{n=1}^{\infty}x_{q_n}^*(x)e_n)|&\le &
  \sum\limits_{n=1}^{\infty}|f(e_n)|\cdot |x^*_{q_n}(x)|\\
  & \le &   \sum\limits_{j=1}^{\infty}\sum\limits_{n\in B_j} |f(e_n)|\cdot
  |x^*_{q_n}(x)|\\
  &\le& C(\sum\limits_{j=1}^{\infty}j\theta_j)\|x\|.
  \end{eqnarray*}
  Therefore the operator $T$ is bounded with $\|T\|\le
  C(\sum\limits_{j=1}^{\infty}j\theta_j)$.

 Finally we prove that the operator $T$ is non-compact.
 The sequence $(x_n)_{n\in\N}$ is bounded, while from the
  biorthogonality of sequence $(x_n,x_n^*)_{n\in\N}$
 it follows that for $i<j$,
 \[
 \|Tx_{q_i}-Tx_{q_j}\|=
 \|\sum\limits_{n=1}^{\infty}x_{q_n}^*(x_{q_i}-x_{q_j})e_n\|=
 \|e_i-e_j\|\ge \frac{1}{2K}\] where $K$ is the basis constant of
 \seq{e}{n}. Therefore $T$ is a non-compact operator.
 \end{proof}

 \begin{proposition}
 There exists a strictly singular non-compact operator $S:T_0\to
 T_0$.
 \end{proposition}
 \begin{proof}[\bf Proof.]
 Let $X,Y$ denote the spaces
 $X=T_0=T[(\mc{A}_{n_j},\frac{1}{m_j})_{j\in\N}]$,
 $Y=T_0'=T[(\mc{A}_{n_{j+1}},\frac{1}{m_j})_{j\in\N}]$
 respectively
 and let $K_0$, $K_0'$ be their standard norming sets.
 From Proposition \ref{Nprop1}
 there exists a block sequence $(x_n^*)_{n\in\N}$ in $X^*=T_0^*$
 which generates a $c_0$ spreading model.
 We also select a bounded block sequence \seq{x}{n} in $T_0$
 with $\ran x_n=\ran x_n^*$ such that $x_n^*(x_n)=1$.
  The standard basis \seq{e}{n} is a normalized Schauder basis of
 the space $Y$, while  for every $j$ and for
 every $\phi\in K_0'$ it holds that
 \[ \#\{n\in\N:\;  |\phi(e_n)|>\frac{1}{m_j}\}\le
 (n_j)^{2}\]
 (the proof of this statement follows similarly with
  those of Lemma \ref{Nlem1} and of
  the claim in the proof of Lemma \ref{Nlem5}).
  Proposition \ref{Nprop12} yields the existence of a strictly
 increasing sequence of integers
 \seq{q}{n} such that the operator $T:T_0\to T_0'$ defined by the rule
 \[T(x)=\sum\limits_{n=1}^{\infty}x^*_{q_n}(x)e_n\] is
 bounded.

   Since
 the norming set $K_0$  of $T_0$ is a subset of the norming set
 $K_0'$ of $T_0'$, the formal identity map
 $I:T_0'\to T_0$ defines a bounded linear operator.
 We show that the operator $I$ is strictly singular.
 Let $Y_1$ be any block subspace of $T_0'$ and let $j\in\N$.
 We may select a block sequence $(y_i)_{i=1}^{n_{j+1}}$ in $Y_1$
 such that the sequence $(Iy_i)_{i=1}^{n_{j+1}}$ is a
 $(3,\frac{1}{m_{j+1}^2})$ R.I.S. in $T_0$ with $\|Iy_i\|_{T_0}\ge 1$,
 and thus $\|y_i\|_{T_0'}\ge 1$. (The technical details for the above
 argument  and the
 definition of R.I.S. are similar to those of Definition \ref{Ndef8}, Lemmas
 \ref{Nlem6}, \ref{Nlem7}, \ref{Nlem8} and Proposition
 \ref{Nprop6}.)
 From the analogue of Proposition \ref{Nprop7} for the space $T_0$
 it follows that
 \[\left\|\frac{Iy_1+Iy_2+\cdots+Iy_{n_{j+1}}}{n_{j+1}}\right\|_{T_0}\le\frac{6}{m_{j+1}}.\]
 On the other hand, selecting $f_i\in K_0'$ with $\ran f_i\subset
 \ran y_i$ and $f_i(y_i)\ge 1$ for $i=1,2,\ldots,n_{j+1}$, the
 functional
 \[  f=\frac{1}{m_j}(f_1+f_2+\cdots+f_{n_{j+1}})  \]
 belongs to the norming set $K_0'$, while its action  yields that
 \[\left\|\frac{y_1+y_2+\cdots+y_{n_{j+1}}}{n_{j+1}}\right\|_{T_0'}\ge\frac{1}{m_{j}}.\]
 Therefore the vector
 $y=\frac{1}{n_{j+1}}\sum\limits_{i=1}^{n_{j+1}}y_i$ belongs to
 the subspace $Y_1$ and
 \[
 \frac{\|Iy\|_{T_0}}{\|y\|_{T_0'}}\le\frac{\frac{6}{m_{j+1}}}{\frac{1}{m_j}}=\frac{6}{m_j}.\]
 Since this procedure may be done for arbitrarily large $j$,  it
 follows that the operator $I$ is strictly singular.

 We define the operator $S:T_0\to T_0$  as the composition
 $S=I\circ T$. The operator $S$ is strictly singular (as $I$ is).
 It is also non-compact, since for the bounded sequence
 $(x_{q_n})_{n\in\N}$ it holds that
 for all $i\neq j$ we have that
 \[  \|S(x_{q_i})-S(x_{q_j})\|_{T_0}=\|e_i-e_j\|_{T_0}=1.\]
 \end{proof}

 \section{The HI space $\eqs_{d}$}\label{Nsec2}

 In this section we  define the space $\eqs_{d}$ and we
 show that it is Hereditarily Indecomposable.
  The unconditional frame we use in the construction of the space $\eqs_{d}$ is the space
  $T_0$ we have constructed in section \ref{Nsec1}.
  For the definition of
 $\eqs_{d}$ we define a Gowers-Maurey type coding function $\sigma$
 and we  define the $n_{2j-1}$ special sequences.

  \begin{definition}\label{Ndef6}[The space $\eqs_{d}$.]
 Let the sequences $(m_j)_{j=1}^{\infty}$, $(n_j)_{j=1}^{\infty}$
 be as in Definition \ref{Ndef4}.
  The set $K_d$  is the minimal subset of $c_{00}(\N)$ satisfying the
  following conditions.
   \begin{enumerate}
 \item[(i)] $\{\pm e_k^*:\;k\in\N\}\subset K_d$.
 \item[(ii)] $K_d$ is  symmetric (i.e. if $f\in K_d$ then  $-f\in K_d$).
 \item[(iii)] $K_d$ is closed under the restriction of its elements
 on intervals of $\N$ (i.e. if $f\in K_d$ and $E$ is an interval of
 $\N$ then $Ef\in K_d$).
 \item[(iv)] For every $j\in \N$, $K_d$ is closed under the $(\mc{A}_{n_{j}},\frac{1}{m_{j}})$
 operation.
 \item[(v)] For every $j\ge 2$, $K_d$ is closed under the
  $(\mc{A}_{n_{2j-1}},\frac{1}{\sqrt{m_{2j-1}}})$
 operation on $n_{2j-1} $  special sequences, i.e. for every $n_{2j-1}$ special sequence\\
 $(f_1,f_2,\ldots,f_{n_{2j-1}})$, with $f_i\in K_d$ for $1\le i\le n_{2j-1}$,  the functional
 $f=\frac{1}{\sqrt{m_{2j-1}}}(f_1+f_2+\cdots+f_{n_{2j-1}})$ also
 belongs to $K_d$.
 \end{enumerate}
 The space $\eqs_d$ is the completion of $(c_{00}(\N),\|\cdot\|_{K_d})$.
 \end{definition}
 The above definition is not complete because we have not yet
 defined the $n_{2j-1}$ special sequences.

  \begin{definition}\label{Ndef5}[The coding function $\sigma$
 and the $n_{2j-1}$ special sequences.]
 Let $\Q_s$ denote the set of all finite sequences
 $(\phi_1,\phi_2,\ldots,\phi_d)$ such that
 $\phi_i\in c_{00}(N)$, $\phi_i\neq 0$ with $\phi_i(n)\in \Q$ for all $i,n$ and
 $\phi_1<\phi_2<\cdots<\phi_d$.
 We fix a pair $\Omega_1,\Omega_2$ of disjoint infinite subsets of
 $\N$.
 From the fact that $\Q_s$ is
 countable we are able to define an injective
 coding function
 $\sigma:\Q_s\to \{2j:\;j\in \Omega_2\}$ such that
 $m_{\sigma(\phi_1,\phi_2,\ldots,\phi_d)}>\max\{\frac{1}{|\phi_i(e_l)|}:\;l\in\supp
 \phi_i,\;i=1,\ldots,d\}\cdot\max\supp \phi_d$.

  Let $j\in \N$. A finite
 sequence $(f_i)_{i=1}^{n_{2j-1}}$ is said to be an
  $n_{2j-1}$ special sequence   provided that
 \begin{enumerate}
 \item[(i)]  $(f_1,f_2,\ldots,f_{n_{2j-1}})\in\Q_s$ and $f_i\in K_d$
 for $i=1,2,\ldots,n_{2j-1}$.
 \item[(ii)] The functional $f_1$ is the result of an
 $(\mc{A}_{n_{2k}},\frac{1}{m_{2k}})$ operation, on a family of functionals belonging to
  of $K_d$, for some for some $k\in \Omega_1$ such that $m_{2k}^{1/2}>n_{2j-1}$
  and for each $1\le i<n_{2j-1}$ the functional $f_{i+1}$ is the
  result of an
  $(\mc{A}_{n_{\sigma(f_1,\cdots,f_i)}},\frac{1}{m_{\sigma(f_1,\cdots,f_i)}})$
   operation on a family of functionals belonging to $K_d$.
  \end{enumerate}
 \end{definition}

 As we have mentioned earlier the weight $w(f)$ of a functional
 $f\in K_d$
  is not unique. However, when we refer to an $n_{2j-1}$ special sequence
 $(f_i)_{i=1}^{n_{2j-1}}$  then, for $2\le i\le n_{2j-1}$,  by
 $w(f_i)$
 we shall always mean  $w(f_i)=m_{\sigma(f_1,\ldots,f_{i-1})}$.

 \begin{proposition}\label{Nprop2}[The tree-like property of $n_{2j-1}$ special
 sequences]\label{treelike}
 Let $\Phi=(\phi_i)_{i=1}^{n_{2j-1}}$, $\Psi=(\psi_i)_{i=1}^{n_{2j-1}}$
 be a pair of distinct $n_{2j-1}$ special sequences. Then
 \begin{enumerate}
 \item[(i)] For $1\le i<l\le n_{2j-1}$ we have that
  $w(\phi_i)\neq  w(\psi_l)$.
 \item[(ii)] There exists $k_{\Phi,\Psi}$ such that
 $\phi_i=\psi_i$ for $i<k_{\Phi,\Psi}$ and
 $w(\phi_i)\neq  w(\psi_i)$ for $i>k_{\Phi,\Psi}$.
 \end{enumerate}
 \end{proposition}
We leave the easy proof to the reader.

 \begin{remark}\label{Nrem3}
  We mention that, since $\sqrt{m_{2j-1}}=m_{2j-2}$ for each $j$,
  (see Definition \ref{Ndef4}) condition (v) in Definition \ref{Ndef6} is equivalent saying
  that   $K_d$ is closed under the
  $(\mc{A}_{n_{2j+1}},\frac{1}{m_{2j}})$ operation on
 $n_{2j+1}$ special sequences for each $j$.

 We call $2j+1$ special functional, every functional of the form
 $Eh$ with $E$ an interval and $h$ the result of a
 $(\mc{A}_{n_{2j+1}},\frac{1}{m_{2j}})$ operation on
 an $n_{2j+1}$ special sequence.

  Let's observe that each $f\in K_d$ is either of the form $f=\pm e_k^*$ or
 there exists $j\in\N$  such that $f$
   takes  the form
  $f=\frac{1}{w(f)}\sum\limits_{i=1}^df_i$ with $d\le n_{2j+1}$
  and $w(f)=m_{2j+1}$ or $w(f)=m_{2j}$.
 \end{remark}

 \begin{remark}
 The trees of functionals $f\in K_d$ and the order $o(f)$ of such  functionals are
 defined in a similar manner as in Definition \ref{Ndef3} and Definition \ref{Ndef26}.
 \end{remark}

 The rest of the present section is devoted to the proof of the HI
 property of the space $\eqs_{d}$.
 We need to introduce the auxiliary spaces $T'$, $T'_{j_0}$.

 \begin{definition}\label{Ndef7}[The auxiliary spaces $T'$, $T'_{j_0}$.]
 Let $T'$ be  mixed Tsirelson space
 \[T'=T[(\mc{A}_{4n_{i}},\frac{1}{m_{i}})_{i\in\N},\;
 (\mc{A}_{4n_{2j+1}},\frac{1}{m_{2j}})_{j\in\N}]\] and we denote by
 $W'$ the standard  norming set corresponding to  this
 space. This means that $W'$
 is the minimal subset of $c_{00}(\N)$ containing $\{\pm
 e_k^*:\;k\in\N\}$ and being closed in the
 $(\mc{A}_{4n_{i}},\frac{1}{m_{i}})_{i\in\N}$
 and in the
 $(\mc{A}_{4n_{2j+1}},\frac{1}{m_{2j}})_{j\in\N}$
 operations.

 We also consider, for each $j_0\in\N$, the auxiliary space
 \[T'_{j_0}=T[(\mc{A}_{4n_{i}},\frac{1}{m_{i}})_{i\in\N},\;
 (\mc{A}_{4n_{2j+1}},\frac{1}{m_{2j}})_{j\neq j_0}]\]
 and we denote by $W'_{j_0}$ its standard norming set.
 \end{definition}

 \begin{lemma} \label{Nlem5}
 Let $j\in \N$ and $f\in W'$. We have that
  \begin{equation*}\label{eq13}
 |f(\frac{1}{n_{2j}}\sum\limits_{k=1}^{n_{2j}}e_k)| \le
 \begin{cases}
 \frac{2}{m_{i}\cdot m_{2j}},\quad &\text{ if }w(f)=m_{i},\;i<2j\\
 \frac{1}{m_{i}},\quad &\text{ if } w(f)= m_{i},\;i\ge 2j
 \end{cases}
 \end{equation*}
 \end{lemma}
  \begin{proof}[\bf Proof.] The case $i\ge 2j$ is obvious. For the case $i<2j$
  we need the following    claim.
  (We shall also use the next claim later in the proofs of Proposition \ref{Nprop4}
  and Lemma \ref{Nlem9}.)
  \begin{claim}
  If $g\in W'$ and $j\in\N$ then
  \[  \#\{k\in\N:\;|g(e_k)|>\frac{1}{m_{2j}}\}\le (4n_{2j-1})^4.  \]
  \end{claim}
  \begin{proof}[\bf Proof.]
  Without loss of generality we may assume that $g(e_k)>\frac{1}{m_{2j}}$ for every
  $k\in\supp g$. Then the functional $g$ has a tree in which appear only the operations
  $(\mc{A}_{4n_i},\frac{1}{m_i})_{1\le i\le 2j-1}$ and
  $(\mc{A}_{4n_{2i+1}},\frac{1}{m_{2i}})_{i<j}$.
   Then (see the proof of Lemma \ref{Nlem1} and  the comments before
   its statement) $\|g\|_q\le 1$ where
   \begin{eqnarray*}
   q&= &\max\Big(\{\log_{m_i}(4n_i):\;1\le i\le 2j-1\}\cup\{\log_{m_{2i}}(4n_{2i+1}):\; i<j\}\Big)\\
   &= &    \log_{m_{2j-2}}(4n_{2j-1}).
   \end{eqnarray*}
   Hence $1\ge \|g\|_q\ge\frac{1}{m_{2j}} \cdot \big(\#(\supp g)\big)^{\frac{1}{q}}$.
   Therefore
   \[ \#(\supp g)\le m_{2j}^q=m_{2j-2}^{4q}=(4n_{2j-1})^4.\]
  \end{proof}

  Let now $f\in W'$ with $w(f)=m_i$, $i<2j$. Then the functional $f$ takes the form
  $f=\frac{1}{m_i}\sum\limits_{r=1}^df_r$ with $f_1<f_2<\cdots<f_d$ in $W'$ and $d\le 4n_{2j-1}$.

  We set $D_r=\{l:\; |f_r(e_l)|>\frac{1}{m_{2j}}\}$ for $r=1,2,\ldots,d$
  and $D=\bigcup\limits_{r=1}^dD_r$. From the claim above we get that
  $\# (D_r)\le (4n_{2j-1})^4$
   for each $r$, thus $\#(D)\le (4n_{2j-1})^5$. Therefore
   \begin{eqnarray*}
   |f(\frac{1}{n_{2j}}\sum\limits_{k=1}^{n_{2j}}e_k)|&  \le &
         |f_{|D}(\frac{1}{n_{2j}}\sum\limits_{k=1}^{n_{2j}}e_k)|
         +|f_{|(\N\setminus D)}(\frac{1}{n_{2j}}\sum\limits_{k=1}^{n_{2j}}e_k)|\\
         & \le & \frac{1}{m_i}\cdot\frac{1}{n_{2j}}\cdot\#(D)+\frac{1}{m_i}\cdot\frac{1}{m_{2j}}\\
         & \le & \frac{1}{m_i} \big( \frac{(4n_{2j-1})^5}{n_{2j}}+\frac{1}{m_{2j}}\big)\\
         & \le & \frac{1}{m_i}\big( \frac{1}{m_{2j}}+\frac{1}{m_{2j}}\big)=\frac{2}{m_i\cdot m_{2j}}.
   \end{eqnarray*}
   \end{proof}

 \begin{lemma}\label{Nlem0002} Let $f\in W_{j_0}'$. Then
      \[|f(\frac{1}{n_{2j_0+1}}\sum\limits_{k=1}^{n_{2j_0+1}}e_k)|\le
      \begin{cases}
  \frac{2}{m_{2j_0+1}m_{i}}\,,\quad &\text{ if }w(f)=m_{i},\; i\le 2j_0\\
  \frac{1}{m_i}\quad ,&\text{ if }w(f)=m_i,\;  i\ge 2j_0+1\\
  \end{cases}  \]
  and therefore $|f(\frac{1}{n_{2j_0+1}}\sum\limits_{k=1}^{n_{2j_0+1}}e_k)|\le \frac{1}{m_{2j_0+1}}$.
   \end{lemma}
   \begin{proof}[\bf Proof.]
   The estimate for $i\ge 2j_0+1$ is obvious.
   For the case $i\le 2j_0$ we shall use the following claim.
   \begin{claim}
    For every $g\in W_{j_0}'$, we have that
  $\#\{k:\;|g(e_k)|>\frac{1}{m_{2j_0+1}}\}\le (4n_{2j_0})^2$.
  \begin{proof}[\bf Proof.]
  Let $g\in W_{j_0}'$. Without loss of generality, we may assume that $g(e_k)>\frac{1}{m_{2j_0+1}}$
  for every $k\in\supp g$. The functional $g$ then, has a tree in which appear  only
  the operations $(\mc{A}_{4n_i},\frac{1}{m_i})_{i\le 2j_0}$ and
  $(\mc{A}_{4n_{2i+1}},\frac{1}{m_{2i}})_{i<j_0}$. Then $\|g\|_q\le 1$, where
  \[q=\max\big(\{ \log_{m_i}(4n_i):\;i\le 2j_0\}\cup \{\log_{m_{2i}}(4n_{2i+1}):\;i<j_0\}\big)
       =\log_{m_{2j_0}}(4n_{2j_0}).\]
  It follows that $1\ge \|g\|_q\ge \frac{1}{m_{2j_0+1}} \#(\supp g)^{1/q}$ therefore
   \[\#(\supp(g)) \le m_{2j_0+1}^q=m_{2j_0+1}^{\log_{m_{2j_0}}(4n_{2j_0})}
     =m_{2j_0}^{2\log_{m_{2j_0}}(4n_{2j_0})}=(4n_{2j_0})^2.\]
    \end{proof}
    \end{claim}

    Let $f\in W_{j_0}'$ with $w(f)=m_i$, $i\le m_{2j_0}$.
   Then the functional $f$ takes the form $f=\frac{1}{m_i}\sum\limits_{r=1}^df_r$
   with $d\le 4n_{2j_0}$. For $r=1,\ldots,d$ we set
   $D_r=\{k:\;|f_r(e_k)|>\frac{1}{m_{2j_0+1}}\}.$
   We also set $D=\bigcup\limits_{r=1}^dD_r$. Then, using the claim,
   we get that $\#(D)\le \sum\limits_{r=1}^d\#(D_r)\le d\cdot (4n_{2j_0})^2\le (4n_{2j_0})^3$.
   Therefore
   \begin{eqnarray*}
   |f(\frac{1}{n_{2j_0+1}}\sum\limits_{k=1}^{n_{2j_0+1}}e_k)| & \le  &
   |f_{|D}(\frac{1}{n_{2j_0+1}}\sum\limits_{k=1}^{n_{2j_0+1}}e_k)|+
   |f_{|(\N\setminus D)|}(\frac{1}{n_{2j_0+1}}\sum\limits_{k=1}^{n_{2j_0+1}}e_k)|
   \\ &   \le &
   \frac{1}{m_i}\cdot \frac{1}{n_{2j_0+1}}\cdot\#(D)+\frac{1}{m_i}\cdot \frac{1}{m_{2j_0+1}}\\
   & \le & \frac{2}{m_i\cdot m_{2j_0+1}}.
   \end{eqnarray*}
   The proof of the lemma is complete.
   \end{proof}

\begin{definition}\label{Ndef8}[R.I.S.]
 A block sequence $(x_{k})_k$ in $\eqs_{d}$ is said to be a
 $(C,\varepsilon)$ rapidly increasing sequence (R.I.S.), if
 $\|x_k\|\le C$ for all $k$, and
  there exists a strictly
 increasing sequence $(j_{k})_k$  of positive integers such that
 \begin{enumerate}
 \item[(a)] $\frac{1}{m_{j_1}}\le \e$ and
 $\frac{1}{m_{j_{k+1}}}\cdot \#\supp x_{k}\le\varepsilon$ for each $k$.
 \item[(b)] For every $k=1,2,\ldots$ and every $f\in K_d$ with
 $w(f)=m_i$, $i<j_{k}$  we have that $|f(x_{k})| \le
 \frac{C}{m_i}$.
 \end{enumerate}
 The sequence $(j_k)_k$ is called the associated sequence of the R.I.S. $(x_k)_k$.
 \end{definition}
 The next proposition is the fundamental tool for
 providing upper bounds
  of the norm for certain vectors in $\eqs_{d}$.

  \begin{proposition}\label{Nprop5}[The basic inequality]
 Let $(x_k)_{k}$ be a $(C,\e)$ R.I.S. in $\eqs_{d}$ with associated sequence $(j_k)_k$, and let
 $(\lambda_k)_k$ be a sequence of scalars. Then for every $f\in K_d$ and every interval $I$ there exists
 a functional \[g\in W'=W[(\mc{A}_{4n_{j}},\frac{1}{m_j})_{j\in\N},(\mc{A}_{4n_{2j+1}},\frac{1}{m_{2j}})_{j\in\N}]\]
  with either $w(g)=w(f)$ of $g=e_r^*$
  such that
  \[
   |f(\sum\limits_{k\in I}\lambda_kx_k)|
  \le C \Big(g(\sum\limits_{k\in I}|\lambda_k|e_k)+\e\sum\limits_{k\in I}|\lambda_k|\Big).
  \]
 Moreover if $f$ is the result of an $(\mc{A}_{n_i},\frac{1}{m_i})$
 operation then either $g=e_r^*$ or $g$ is the result of  an
 $(\mc{A}_{4n_i},\frac{1}{m_i})$operation.

    If we additionally assume that for some $2j_0+1<j_1$ we have that
  for every subinterval $J$ of $I$ and every $2j_0+1$ special functional $f$ it holds that
  \begin{equation}\label{n74}
   |f(\sum\limits_{k\in J}\lambda_kx_k)| \le
  C \Big(\max\limits_{k\in J}|\lambda_k|+\e\sum\limits_{k\in J}|\lambda_k|\Big).
  \end{equation}
 then we may select the functional $g$ to be  in
 \[W'_{j_0}=W[(\mc{A}_{4n_{j}},\frac{1}{m_j})_{j\in\N},
        (\mc{A}_{4n_{2j+1}},\frac{1}{m_{2j}})_{j\neq j_0}].\]
   \end{proposition}
  \begin{proof}[\bf Proof.]
  We first treat the case that for some $j_0$, the additional assumption \eqref{n74}
   in the statement of the proposition
  is satisfied. We proceed by induction on the order $o(f)$ of the functional $f$.

  If $o(f)=1$, i.e. if  $f=\pm e_r^*$, then we set $g=e_k^*$ for the unique $k\in I$ for
  which $r\in\ran(x_k)$ if such a $k$ exists, otherwise we set $g=0$.

  Suppose now that the result holds for every functional in $K_d$ with order less than $q$ and
  consider $f\in K_d$ with $o(f)=q$. Then
  \[f=\frac{1}{w(f)}(f_1+f_2+\cdots +f_d)  \]
  where $f_1<f_2<\cdots<f_d$ are in $K_d$ with $o(f_i)<q$, and either $w(f)=m_j$ and $d\le n_j$,
  or $f$ is a $2j+1$ special functional (then $w(f)=\sqrt{m_{2j+1}}=m_{2j}$ and $d\le n_{2j+1}$).
    We distinguish four cases.

   {\bf Case 1.} $f$ is a $2j_0+1$ special functional.\\
   We
   choose $k_0\in I$  with $|\lambda_{k_0}|=\max\limits_{k\in I}|\lambda_k|$ and we set
   $g=e_{k_0}^*$. Then from our assumption \eqref{n74} it follows that
   \begin{eqnarray*}
   |f(\sum\limits_{k\in I}\lambda_k x_k)|& \le & C(\max_{k\in I}|\lambda_k |
                                  +\e \sum\limits_{k\in I}|\lambda_k|\Big)\\
                  & \le & C\Big(g(\sum\limits_{k\in I}|\lambda_k |e_k)+
                   \e\sum\limits_{k\in I}|\lambda_k|\Big).
   \end{eqnarray*}

   {\bf Case 2.} $w(f)<m_{j_k}$ for all $k\in I$ and $f$ is not a $2j_0+1$ special functional.\\
   For $i=1,\ldots,d$ we set $E_i=\ran(f_i)$, and
   \[ I_i=\{k\in I:\; \ran(x_k)\cap E_i\neq\emptyset\mbox{ and }\ran(x_k)\cap E_{i'}=\emptyset
   \mbox{ for all } i'\in I\setminus\{i\}\}.\]
   We also set
   \[ I_0=\big\{k\in I:\; \ran(x_k)\cap E_i\neq\emptyset\mbox{ for at least two }i\in \{1,\ldots,\}\big\}\]
   and $I'=I\setminus\bigcup\limits_{i=0}^dI_i$.

   We observe that $|I_0|\le d$. For each $k\in I_0$ assumption (b) in the definition of R.I.S.
   yields that
   \begin{equation}\label{n75}
   |f(x_k)|\le \frac{C}{w(f)}.
   \end{equation}
   Observe also, that for each $i=1,\ldots,d$, $I_i$ is a subinterval of $I$, hence our inductive
   assumption yields that there exists $g\in W_{j_0}'$ with $\supp g_i\subset I_i$ such that
   \begin{equation}\label{n76}
   |f_i(\sum\limits_{k\in I_i}\lambda_kx_k)|
  \le C \Big(g_i(\sum\limits_{k\in I_i}|\lambda_k|e_k)+\e\sum\limits_{k\in I_i}|\lambda_k|\Big).
  \end{equation}
   The family $\{I_1,\ldots,I_d\}\cup\{\{k\}:\;k\in I_0\}$ consists of pairwise disjoint intervals
   and has cardinality less than or equal to $2d$. We set
   \[ g=\frac{1}{w(f)}\Big(\sum\limits_{i=1}^d g_i+\sum\limits_{k\in I_0}e_k^*\Big).\]
   Then $g\in W_{j_0}'$, $\supp g\subset I$, while from \eqref{n75},\eqref{n76} we get that
   \begin{eqnarray*}
   |f(\sum\limits_{k\in I}\lambda_kx_k)| & \le &
   \sum\limits_{k\in I_0}|\lambda_k||f(x_k)|+
   \frac{1}{w(f)}\sum\limits_{i=1}^d|f_i(\sum\limits_{k\in I_i}\lambda_kx_k)|\\
   & \le & \sum\limits_{k\in I_0}|\lambda_k|\frac{C}{w(f)}+\frac{1}{w(f)}
   \sum\limits_{i=1}^dC\Big(g_i(\sum\limits_{k\in I_i}|\lambda_k|e_k)+\e\sum\limits_{k\in I_i}|\lambda_k|\Big)\\
    & \le & C\Big(g(\sum\limits_{k\in I}|\lambda_k|e_k)+\e\sum\limits_{k\in I}|\lambda_k|\Big).
   \end{eqnarray*}

   {\bf Case 3.} $m_{j_{k_0}}\le w(f)<m_{j_{k_0}+1}$ for some $k_0\in I$.\\
   In this case, for $k\in I$ with $k<k_0$ we have that $m_{j_{k+1}}\le m_{j_{k_0}}\le w(f)$, hence,
   using assumption (a) in the definition of R.I.S. it follows that
   \begin{equation}\label{n77}
   |f(x_k)|\le \frac{1}{w(f)}\|x_k\|_{\ell_1}\le \frac{1}{m_{j_{k+1}}}\cdot C\cdot\#\supp(x_k)\le C\e.
   \end{equation}
   For $k\in I$ with $k>k_0$,  from assumptions (a), (b) in the definition of R.I.S. we get that
    \begin{equation}\label{n78}
   |f(x_k)|\le \frac{C}{w(f)}\le \frac{C}{m_{j_1}}\le C\e.
   \end{equation}
   Thus, setting $g=e_{k_0}^*$ and using \eqref{n77}, \eqref{n78} we get that
   \begin{eqnarray*}
    |f(\sum\limits_{k\in I}\lambda_k x_k) | & \le&
      |\lambda_{k_0}||f(x_{k_0})|+\sum\limits_{\genfrac{}{}{0pt}{}{k\in I}{k\neq k_0}}|\lambda_k| |f(x_k)|\\
      & \le&  |\lambda_{k_0}| C+\sum\limits_{\genfrac{}{}{0pt}{}{k\in I}{k\neq k_0}}|\lambda_k|C\e\\
      & \le&  C\Big(g(\sum\limits_{k\in I}|\lambda_k|e_k+\e\sum\limits_{k\in I}|\lambda_k|\Big)
    \end{eqnarray*}

   {\bf Case 4.} $m_{j_{k+1}}\le w(f)$ for all $k\in I$.\\
   In this case, as in Case 3, we get that $|f(x_k)|\le C\e$ for all $k\in I$ so we may set $g=0$.

   This completes the proof in the case we have made the additional assumption about $j_0$.
   When no assumption about $j_0$ is made, the induction is similar to the previous one, with the
   only difference concerning  Case 2, where  we include $f$ which is a $2j_0+1$ special functional
   (thus Case 1 does not appear). In each inductive step the resulting functional $g$ belongs to
   $W'$.
      \end{proof}

 From  Proposition \ref{Nprop5} and  Lemma \ref{Nlem5}
 we conclude the following.
  \begin{proposition}\label{Nprop7}
 Let $(x_k)_{k=1}^{n_{2j}}$ be a $(C,\e)$ R.I.S.  with
 $\e\le\frac{1}{m_{2j}^2}$.
 Let also $f\in K_d$. Then
  \[|f(\frac{1}{n_{2j}}\sum\limits_{k=1}^{n_{2j}}x_{k})| \le
 \begin{cases}
 \frac{3C}{m_{2j}m_{i}}\,,\quad &\text{ if }w(f)=m_{i},\; i<2j\\
 \frac{C}{m_{i}}+C\e\,,&\text{ if }w(f)=m_i,\;  i\ge 2j\\
 \end{cases}\]
 In particular
 $\|\frac{1}{n_{2j}}\sum\limits_{k=1}^{n_{2j}}x_{k}\|
 \le\frac{2C}{m_{2j}}$.
  \end{proposition}

 \begin{definition}\label{Ndef9} A vector $x\in \eqs_d$ is said to be a $C-\ell^k_1$
 average if $x$ takes the form
 $x=\frac{1}{k}\sum\limits_{i=1}^kx_i$,
 with $\|x_i\|\le C$ for each $i$, $x_1<\cdots<x_k$ and $\|x\|\ge 1$.
 \end{definition}

 \begin{lemma}\label{Nlem6} Let $Y$ be a block subspace of $\eqs_d$ and let $k\in \N$ .
 Then there exists a vector $x\in Y$ which is a $2-\ell_1^k$
 average.
 \end{lemma}
 For a proof we refer to \cite{ArTo} Lemma II.22.

 \begin{lemma}\label{Nlem7} If $x$ is a $C-\ell_1^k$ average, $d\le k$ and $E_1<\cdots
 <E_d$ is a sequence of intervals then
 $\sum\limits_{i=1}^d\|E_ix\|\le C(1+\frac{2d}{k})$.
 In particular
 if $x$ is a $C-\ell_{1}^{n_{2j}}$ average then for every
 $f\in K_d$ with $w(f)=m_{i}$, $i<2j$ we have that $|f(x)|\le
 \frac{1}{m_{i}}C(1+\frac{2n_{2j-1}}{n_{2j}})\le
 \frac{3C}{2}\frac{1}{w(f)}$.
 \end{lemma}

  For a proof we refer to \cite{ArTo} Lemma II.23. The next lemma is a direct consequence
 of Lemma \ref{Nlem7}.

 \begin{lemma}\label{Nlem8} Let \seq{x}{k} be a block sequence in
 $\eqs_{d}$
 such that each $x_k$ is a $C-\ell_{1}^{n_{2l_k}}$ average, where
 \seq{l}{k} is a strictly increasing
  sequence of integers, and let $\e>0$.
 Then there exists a subsequence of \seq{x}{k} which is a
 $(\frac{3C}{2},\e)$ R.I.S.
 \end{lemma}

 \begin{proposition}\label{Nprop6}[Existence of R.I.S.]
 For every $\e>0$ and every block subspace $Z$ of $\eqs_{d}$
 there exists a $(3,\e)$ R.I.S. $(x_k)_{k\in\N}$ in $Z$ with $\|x_k\|\ge 1$.
 \end{proposition}
 \begin{proof}[\bf Proof.] It follows from Lemma \ref{Nlem6} and Lemma
 \ref{Nlem8}.
 \end{proof}

 \begin{definition}\label{Ndef10}[Exact pairs.]
  A pair $(x,\phi)$ with $x\in \eqs_{d}$
 and $\phi\in K_{d}$ is said to be a $(C,2j)$ exact pair
 (where  $C\ge 1$,  $j\in\N$)
 if the  following conditions are satisfied:
 \begin{enumerate}
 \item[(i)] $1\le \|x\|\le C$, for every
 $\psi\in K_{d}$ with $w(\psi)<m_{2j}$,  we have
 that $|\psi(x)|\le\frac{3C}{w(\psi)}$, while
 for $\psi\in K_{d}$ with $w(\psi)>m_{2j}$,
  $|\psi(x)|\le\frac{C}{m_{2j}^2}$.
 \item[(ii)] $\phi\in K_d$ with $w(\phi)=m_{2j}$.
 \item[(iii)] $\phi(x)=1$ and $\ran x=\ran \phi$.
 \end{enumerate}
 \end{definition}

 \begin{proposition}\label{Nprop8}
 Let  $j\in\N$. Then for every block subspace $Z$ of
 $\eqs_{d}$,  there exists  a $(6,2j)$ exact pair $(x,\phi)$ with $x\in Z$.
 \end{proposition}
 \begin{proof}[\bf Proof.]
 From Proposition \ref{Nprop6} there exists $(x_{k})_{k=1}^{n_{2j}}$ a
 $(3,\e)$-R.I.S. in $Z$ with $\e\le \frac{1}{2m_{2j}^{3}}$ and $\|x_k\|\ge 1$. Choose
 $x_{k}^{*}\in K_{d}$ with $x_{k}^{*}(x_{k})\ge 1$ and $\ran
 x_{k}^{*}=\ran x_{k}$. Then Proposition \ref{Nprop7} yields
 that for some $\theta$ with $\frac{1}{6}\le \theta\le 1$,
 \[\big(\theta\frac{m_{2j}}{n_{2j}}\sum_{k=1}^{n_{2j}}x_{k},
 \frac{1}{m_{2j}}\sum_{k=1}^{n_{2j}}x_{k}^{*}\big)\] is a $(6,2j)$
 exact pair.
 \end{proof}

 \begin{definition}\label{Ndef11}[Dependent sequences.]
 A double sequence $(x_k,x_k^*)_{k=1}^{n_{2j+1}}$ with $x_k\in
 \eqs_{d}$ and $x_k^*\in K_{d}$ is said to be a $(C,2j+1)$
 dependent sequence  if there exists a sequence
 $(2j_k)_{k=1}^{n_{2j+1}}$ of even integers such that the following
 conditions are fulfilled:
 \begin{enumerate}
  \item[(i)] $(x^*_k)_{k=1}^{n_{2j+1}}$ is an $n_{2j+1}$ special
               sequence with $w(x^*_{k})=m_{2j_{k}}$
               for all $1\le k\le n_{2j+1}$.
  \item[(ii)] Each $(x_{k},x_{k}^*)$ is a $(C,2j_{k})$ exact pair.
  \end{enumerate}
 \end{definition}

 \begin{remark}\label{Nrem4}
 It follows easily, that if $(x_k,x_k^*)_{k=1}^{n_{2j+1}}$ is a
 $(C,2j+1)$ dependent sequence then the sequence
 $(x_k)_{k=1}^{n_{2j+1}}$ is a $(3C,\e)$ R.I.S. where
 $\e=\frac{1}{n_{2j+1}^2}$.
 \end{remark}

 \begin{proposition}\label{Nprop9} Let
   $j\in\N$. Then for every pair of block subspaces $Z,W$ of
  $\eqs_{d}$ there exists  a  $(6,2j+1)$ dependent sequence
  $(x_k,x_k^*)_{k=1}^{n_{2j+1}}$ with  $x_{2k-1}\in Z$
    and $x_{2k}\in W$ for all $k$.
 \end{proposition}
\begin{proof}[\bf Proof.] It follows easily from an inductive application of
              Proposition \ref{Nprop8}.
\end{proof}

 We need the next lemma  in order
 to apply Proposition \ref{Nprop5} with the additional assumption.

\begin{lemma}\label{Nlem0000}
  Let $(x_k,x_k^*)_{k=1}^{n_{2j+1}}$ be a
 $(C,2j+1)$ dependent sequence. Then for every $2j+1$ special functional $f$ and every
 subinterval $I$ of $\{1,2,\ldots, n_{2j+1}\}$ we have that
  $|f(\sum\limits_{k\in I}(-1)^{k+1}x_k)|\le C$.
    \end{lemma}
    \begin{proof}[\bf Proof.]
    The functional $f$ takes the form
    \[ f=\frac{1}{m_{2j}}(Ex_t^*+x_{t+1}^*+\cdots+x_{r-1}^*+f_r+f_{r+1}+\cdots+f_d)\]
    where $(x_1^*,x_2^*,\ldots,x_{r-1}^*,f_r,f_{r+1},\ldots,f_{n_{2j+1}})$
    is an $n_{2j+1}$ special sequence
    with\\ $w(f_r)=w(x_r^*)$, $f_r\neq x_r^*$, $E$ is an interval
    of the form $E=[m,\max\supp x_t^*]$ and
      $d\le n_{2j+1}$.

      Using  the definitions of dependent sequences and exact
       pairs we obtain the following.\\
      For  $k<t$ we have that $f(x_k)=0$.\\
      For $k=t$, $|f(x_t)|=\frac{1}{m_{2j}}|Ex_t^*(x_t)|\le \frac{1}{m_{2j}}\cdot \|x_t\|\le \frac{C}{m_{2j}}$.\\
      For $t<k<r$, we get that $f(x_k)=\frac{1}{m_{2j}}x_k^*(x_k)=\frac{1}{m_{2j}}$.\\
      For the case $k=r$ we shall say later.\\
      Let $k$ with $r<k\le n_{2j+1}$. For $i\le r-1$ we have that $\ran(x_i^*)\cap \ran x_k=\emptyset$
      thus $x_i^*(x_k)=0$. Also, the injectivity of the coding function $\sigma$ yields that
      $w(f_i)\neq m_{2j_k}=w(x_k^*)$ for $r\le i\le d$. Setting
      \[  J_k^-=\{i:\; w(f_i)<m_{2j_k}\} \mbox{\quad and \quad} J_k^+=\{i:\; w(f_i)> m_{2j_k}\} \]
      we get that
      \begin{eqnarray*}
      |f(x_k)| & \le & \frac{1}{m_{2j}}\big(\sum\limits_{i\in J_k^-}|f_i(x_k)|+
                                          \sum\limits_{i\in J_k^+}|f_i(x_k)|\big)\\
  & \le & \frac{1}{m_{2j}}\big(\sum\limits_{i\in J_k^-}\frac{3C}{w(f_i)}+
                                          \sum\limits_{i\in J_k^+}\frac{C}{m_{2j_k}^2}\big)\\
      & \le & \frac{C}{m_{2j}}\big( \frac{4}{w(x_1^*)}+n_{2j+1}\cdot\frac{1}{m_{2j_1}^2}\big)\\
         & \le & \frac{C}{m_{2j}}(\frac{4}{n_{2j+1}^2}+n_{2j+1}\cdot\frac{1}{n_{2j+1}^4}\big)\le
         \frac{5C}{m_{2j}}\cdot\frac{1}{n_{2j+1}^2}.
      \end{eqnarray*}
      For $k=r$ using similar arguments it follows that $|f(x_r)|\le \frac{2C}{m_{2j}}$.

     We set $I_1=I\cap\{t\}$,  $I_2=I\cap\{t+1,\ldots, r-1\}$,
     $I_3=I\cap\{r\}$,\\ $I_4=I\cap\{r+1,\ldots, n_{2j+1}\}$ and we conclude that
     \begin{eqnarray*}
     |f(\sum\limits_{k\in I}(-1)^{k+1}x_k)|&\le& \sum\limits_{k\in I_1}|f(x_k)|+
                              |f(\sum\limits_{k\in I_2}(-1)^{k+1}x_k)|\\
                              && \qquad\qquad+
                  \sum\limits_{k\in I_3}|f(x_k)|+
                  \sum\limits_{k\in I_4}|f(x_k)|\\
      &\le&  \frac{C}{m_{2j}}+\frac{1}{m_{2j}}+\frac{2C}{m_{2j}}
      +n_{2j+1}\cdot\frac{5C}{m_{2j}}\cdot\frac{1}{n_{2j+1}^2}            \\
      & \le &  C.
     \end{eqnarray*}
    \end{proof}

 \begin{proposition}\label{Nprop10} Let $(x_k,x_k^*)_{k=1}^{n_{2j+1}}$ be a
 $(C,2j+1)$ dependent sequence. Then
  \[
 \|\frac{1}{n_{2j+1}}\sum\limits_{k=1}^{n_{2j+1}}(-1)^{k+1}x_k\|
 \le\frac{4C}{m_{2j+1}}.
 \]
 \end{proposition}
 \begin{proof}[\bf Proof.]
 The sequence $(x_k)_{k=1}^{n_{2j+1}}$ is a $(3C,\e)$ R.I.S. for $\e=\frac{1}{n_{2j+1}^2}$
 (Remark \ref{Nrem4}).
 It follows from Lemma \ref{Nlem0000} that the additional assumption of
  Proposition \ref{Nprop5}   concerning the number $j_0=j$ and the sequence
  $(\frac{(-1)^{k+1}}{n_{2j+1}})_{k=1}^{n_{2j+1}}$ is fulfilled.
   Thus applying  Proposition \ref{Nprop5} and
    Lemma  \ref{Nlem0002} we get that for every $f\in K_d$ there exists $g\in W'_j$
    such that
   \begin{eqnarray*}
   |f(\frac{1}{n_{2j+1}}\sum\limits_{k=1}^{n_{2j+1}}(-1)^{k+1}x_k)| & \le  &
    3C\big(g(\frac{1}{n_{2j+1}}\sum\limits_{k=1}^{n_{2j+1}}e_k)+\frac{1}{n_{2j+1}^2}\big))  \\
         & \le & 3C (\frac{1}{m_{2j+1}}+\frac{1}{n_{2j+1}^2})\\
      & \le & \frac{4C}{m_{2j+1}}.
      \end{eqnarray*}
   This completes the proof of the proposition.
   \end{proof}

 \begin{theorem}\label{Nth1}
 The space $\eqs_{d}$ is a reflexive HI space.
 \end{theorem}
 \begin{proof}[\bf Proof.]
 The Schauder basis \seq{e}{n} of the space $\eqs_d$ is boundedly complete
 and shrinking (this follows by similars arguments with the corresponding
 result in \cite{GM1}). Therefore $\eqs_{d}$ is a reflexive space.

 Let $Z,W$ be a pair of infinite dimensional subspaces of
 $\eqs_d$. We shall show that for every $\e>0$ there exist
 $z\in Z$, $w\in W$ with $\|z-w\|<\e\|z+w\|$. It is easy to check that this
 yields the HI property of $\eqs_d$. From the well known
 gliding hump argument we may assume that $Z,W$ are block
 subspaces. Then for $j\in\N$, using Proposition \ref{Nprop9},
 we select  $(x_k,x_k^*)_{k=1}^{n_{2j+1}}$ a $(6,2j+1)$
 dependent sequence with
 $x_{2k-1}\in Z$ and $x_{2k}\in W$ for all $k$.
 From Proposition \ref{Nprop10} we get that
 \[\|\frac{1}{n_{2j+1}}\sum\limits_{k=1}^{n_{2j+1}}(-1)^{k+1}x_k\|
 \le\frac{24}{m_{2j+1}}.\]
  On the other hand, since
 $(x_k^*)_{k=1}^{n_{2j+1}}$ is an $n_{2j+1}$ special sequence,
 the functional
 $f=\frac{1}{\sqrt{m_{2j+1}}}\sum\limits_{k=1}^{n_{2j+1}}x_k^*$
 belongs to the norming set $K_d$, while the action of $f$ on the
 vector
 $\frac{1}{n_{2j+1}}\sum\limits_{k=1}^{n_{2j+1}}x_k$
 yields that
 \[ \|\frac{1}{n_{2j+1}}\sum\limits_{k=1}^{n_{2j+1}}x_k\|\ge \frac{1}{\sqrt{m_{2j+1}}}.  \]

 Thus setting $z=\sum\limits_{k=1}^{n_{2j+1}/2}x_{2k-1}$ and
 $w=\sum\limits_{k=1}^{n_{2j+1}/2}x_{2k}$ we get that
 $z\in Z$, $w\in W$ and $\|z-w\|\le\frac{24}{\sqrt{m_{2j+1}}}\|z+w\|$
 which for sufficiently large $j$ yields the desired result.

 Therefore the space $\eqs_d$ is HI.
 \end{proof}

\section{A class of bounded diagonal operators on $\eqs_{d}$}\label{Nsec3}

 In this section we present the construction of a class of bounded
 diagonal operators on the space $\eqs_d$. These operators are of the form
 $\sum\limits_k\lambda_k D_{j_k}$ where $\{j_k:\;k\in\N\}$ is a lacunary
 set and \seq{\lambda}{k} is any bounded sequence of real numbers. Each
 $D_{j_k}$ is of the form
 $D_{j_k}(x)=\frac{1}{m_{j_k}}\sum\limits_{i=1}^{p_k}I^{j_k}_ix$.
 We pass to the details of the construction.

  Let $\{I_i^j:\;  1\le i\le p_j,\;j=1,2,\ldots\}$ be
 any  family of intervals of $\N$ such that, for every $j$
 \begin{equation*}
 I_1^j<I_2^j<\cdots <I_{p_j}^j<I_1^{j+1}.
 \end{equation*}
 For each $j\in {\mathbb N}$, we define the diagonal operator
 $D_j:\eqs_{d}\rightarrow \eqs_{d}$ by the rule
 \begin{equation*}
 D_j(x)=\frac{1}{m_j}\sum_{i=1}^{p_j}I_i^jx.
 \end{equation*}
 We also define
 \[\alpha_j(x)=\frac{1}{m_j}\sum_{i=1}^{p_j}\| I_i^jx\|\]
 and we observe that for every $j\in\N$ and $x\in \eqs_{d}$ we
 have that
 \[  \|D_jx\|\le \alpha_j(x)\le \|x\|.\]
 Indeed, the left inequality is obvious
 while, in order to prove the right one,
 for $i=1,\ldots ,p_j$, we select $\phi_i\in K_d$ such that
 $\supp(\phi_i)\subset I_i^j$ and $\phi_i(x)=\|
 I_i^jx\|$. Then, $\phi
 =\frac{1}{m_j}\sum\limits_{i=1}^{p_j}\phi_i\in K_d$, thus
 \[
 \alpha_j(x)= \frac{1}{m_j}\sum_{i=1}^{p_j}\| I_i^jx\|=\phi(x)\le
 \|x\|.\]

 \begin{lemma}\label{Nlem3}
 Let $L\subset {\mathbb N}$ with $\#L\le\min L$.
 Then, for every  $x\in \eqs_d$, we have that
 \[\sum_{j\in L}\alpha_j(x)\le\| x\|.\]
 \end{lemma}
 \begin{proof}[\bf Proof.] Let $L=\{ j_1,j_2,\ldots ,j_s\}$ with
 $s\le j_1<j_2<\cdots <j_s$. For every $k=1,\ldots ,s$ and
 $i=1,\ldots ,p_{j_k}$ we choose $\phi_i^k\in K_d$ such that
 $\supp(\phi_i^k)\subset I_i^{j_k}$ and $\phi_i^k(x)=\|
  I_i^{j_k}x\|$. Then, for every $k=1,\ldots ,s$, we have that
 $\phi^k=\frac{1}{m_{j_k}}\sum\limits_{i=1}^{p_{j_k}}\phi_i^k\in K_d$ and
 $\phi^k(x)=\alpha_{j_k}(x)$. Moreover, as in the proof of
 Proposition \ref{Nprop1},  the functional
 $f=\sum\limits_{k=1}^s\phi^k$ takes the form
 \[f=\frac{1}{m_{j_1}}\sum_{l=1}^{n_{j_1}}\psi_l\]
  with $(\psi_l)_{l=1}^{n_{j_1}}$ being successive members of $K_d$, hence $f\in K_d$.
   It follows that
 \[\sum_{k=1}^s\alpha_{j_k}(x)=f(x)\le\| x\|.\]
 \end{proof}

 \begin{proposition}\label{Nprop4}
 Let  $M=\{j_k:\;j\in\N\}$ be a subset of $\N$
 such that for every $k$ the following conditions are satisfied:
 \begin{enumerate}
 \item[(i)] $m_{j_{k+1}}\ge 2^{k+1}\cdot n_{{j_k}+1}$.
 \item[(ii)] $m_{j_{k+1}} \ge 2^k\cdot\max I_{p_{j_k}}^{j_k}$.
 \item[(iii)]  $j_k>n_{2k}$.
 \end{enumerate}
 Then for every
 $(\lambda_k)_{k\in\N}\in\ell_{\infty}(\N)$, the operator
 $\sum\limits_k\lambda_k D_{j_k}$  is bounded and strictly singular
  with \[\|\sum\limits_k\lambda_k D_{j_k}\|\le C_0\cdot \sup\limits_k|\lambda_k|\] where
  $C_0=3+\sum\limits_{i=1}^{\infty }\frac{i+1}{m_{2i}}$.
 \end{proposition}

    We divide   the proof of Proposition
       \ref{Nprop4} in several steps. The main step
      is done in the following proposition.

 \begin{proposition}\label{Nprop3}
 For every $f\in K_d$ and every interval $I$ there exists $g\in
 W'$ (recall that $W'$ is the norming set of the space
 $T'=T[(\mc{A}_{4n_{i}},\frac{1}{m_{i}})_{i\in\N},\;
 (\mc{A}_{4n_{2j+1}},\frac{1}{m_{2j}})_{j\in\N}]$, see Definition
 \ref{Ndef7}) having nonnegative coordinates, with $\supp g\subset I$, such that for every $x\in
 \eqs_d$ is holds that
  \[ |f(D_{j_k}x)|\le \alpha_{j_k}(x) g(e_k)+\frac{1}{2^k}\|x\|
  \] for all $k\in I$ with the potential exception for
  $k\in\{k_0,k_0+1\}$ where $k_0+1<\supp g$.
   \end{proposition}

 \noindent For the proof, we need the following Lemma.

 \begin{lemma}\label{Nlem4} Let $k\in \N$,  $\phi\in K_d$ and $x\in \eqs_d$.
 \begin{enumerate}
 \item[(i)] If $w(\phi)\le m_{j_{k-1}}$ then
 \[ |\phi(D_{j_k}(x)) |=\left |\phi\left
 (\frac{1}{m_{j_k}}\sum_{i=1}^{p_{j_k}}I_i^{j_k}x\right )\right
 |\le\frac{1}{w(\phi )}\alpha_{j_k}(x)+\frac{1}{2^k}\| x\|.\]
 \item[(ii)]
 If $w(\phi)\ge m_{j_{k+1}}$ then
 $|\phi(D_{j_k}(x))|\le\frac{1}{2^k}\|x\|$.
 \end{enumerate}
 \end{lemma}
 \begin{proof}[\bf Proof.]

 (i) Let $\phi\in K_d$ with $w(\phi)\le m_{j_{k-1}}$.
 Then $\phi =\frac{1}{w(\phi
 )}\sum\limits_{l=1}^d\phi_l$ where, for some $j\in {\mathbb N}$,
 either $w(\phi )=m_j$ or $w(\phi )=\sqrt{m_j}$ and $d\le n_j$.
 Since $w(\phi )\le m_{j_{k-1}}$, in either case we get that
 $d\le n_{j_{k-1}+1}$.

 For $l=1,\ldots ,d$, we set
 \[ R_l  = \big\{ i :\;\ran(\phi_l)\cap I_i^{j_k}\neq\emptyset \mbox{ and }
                       \ran(\phi_{l'})\cap I_i^{j_k}=\emptyset  \mbox{ for }l'\neq l
                       \big\}.
      \]
 We also set
 \[
 A=
  \big\{ i\in\{ 1,\ldots ,p_{j_k}\}:\;\ran(\phi_l)\cap
  I_i^{j_k}\neq\emptyset\mbox{ for at least two }l\big\}.
 \]
  It is easy to see that $\#A\le d\le
 n_{j_{k-1}+1}$. For every $x\in \eqs_d$ we get that
 \begin{eqnarray*}
 |\phi (D_{j_k}(x))| &=&  \big|\frac{1}{m_{j_k}}\sum_{i\in A}\phi
 (I_i^{j_k}x)+\frac{1}{w(\phi)}\frac{1}{m_{j_k}}\sum_{l=1}^d\phi_l
 \big(\sum_{i\in R_l}I_i^{j_k}x \big)\big|\\
 &\le & \frac{1}{m_{j_k}}\sum_{i\in A}\| I_i^{j_k}x\|
 +\frac{1}{w(\phi )}\frac{1}{m_{j_k}}\sum_{i=1}^{p_{j_k}}\|
 I_i^{j_k}x\|\\
 &\le & \frac{1}{m_{j_k}}n_{j_{k-1}+1}\| x\| +\frac{1}{w(\phi
 )}\alpha_{j_k}(x).
 \end{eqnarray*}
 From property (i) of the sequence $(j_k)_{k=1}^{\infty }$, we get
 that
 \[|\phi (D_{j_k}(x))|\le \frac{1}{2^k}\| x\|+\frac{1}{w(\phi
 )}\alpha_{j_k}(x).\]

 (ii) Let now $\phi\in K_d$ with $w(\phi)\ge m_{j_{k+1}}$ and
  $x\in \eqs_d$. We have that
  \begin{eqnarray*}
\big | \phi \big ( D_{j_k}(x)\big )\big | &= & \left |\phi\left
(\frac{1}{m_{j_k}}\sum_{i=1}^{p_{j_k}}I_i^{j_k}x\right )\right |
 \le  \frac{1}{m_{j_k}}\frac{1}{w(\phi)}\left\|
\sum_{i=1}^{p_{j_k}}I_i^{j_k}x\right\|_{\ell_1}\\
&\le & \frac{1}{w(\phi)} \cdot\max I_{p_{j_k}}^{j_k}\cdot\|
x\|_{\infty }
\le  \frac{1}{m_{j_{k+1}}}\cdot\max I_{p_{j_k}}^{j_k}\cdot\| x\|\\
 &\le &\frac{1}{2^k}\|x\|
\end{eqnarray*}
where the last inequality follows from property (ii) of the
sequence $(j_k)_{k=1}^{\infty}$.
 \end{proof}

  \begin{proof}[\bf Proof of Proposition \ref{Nprop3}.]
 For each $k=1,2,\ldots$ let $I^k$ be the minimal interval
 containing $\bigcup\limits_{i=1}^{p_{j_k}}I_i$. We proceed to the
 proof by induction on the order $o(f)$ of the functional $f$.

 If $o(f)=1$, i.e. if $f=\pm e_r^*$, then, if $r\in I^k$ for some
 $k\in I$ we set $g=e_k^*$, otherwise we set $g=0$.
 Suppose now that the conclusion holds for every functional in
 $K_d$ having order less than $q$ and consider $f\in K_d$ with
 $o(f)=q$. Then
 $f=\frac{1}{w(f)}(f_1+f_2+\cdots+f_d)$ with $o(f_i)<q$ for each
 $i$, while either $w(f)=m_j$ and $d\le n_j$ for some $j$, or $f$
 is a $2j+1$ special functional, in which case
 $w(f)=\sqrt{m_{2j+1}}=m_{2j}$ and $d\le n_{2j+1}$.

 For $i=1,\ldots, d$ we set
 \[ I_i=\big\{k\in I:\; \ran(f_i)\cap I^k\neq\emptyset \mbox{ and }
                    \ran(f_{i'})\cap I^k=\emptyset\mbox{ for
                    }i'\in I\setminus\{i\}\big\}.\]
 We also set
 \[ I_0=\big\{k\in I:\;\ran(f_i)\cap I^k\neq\emptyset\mbox{ for at
 least two }i\big\}\]
 and we observe that $\#I_0\le d$.

 Let now $k_0\in \N$ be such that $m_{j_{k_0}}\le w(f)<
 m_{j_{k_0+1}}$ (the modifications in the rest of the proof are
 obvious if no such $k_0$ exists, i.e. if $w(f)<m_{j_1}$).
 For $k<k_0$, Lemma \ref{Nlem4} (ii) yields that $|f(D_{j_k}(x))|\le
 \frac{1}{2^k}\|x\|$ for every $x\in \eqs_d$, while for $k>k_0+1$, Lemma
 \ref{Nlem4} (i) yields that $|f(D_{j_k}(x))|\le
 \frac{1}{w(f)}\alpha_{j_k}(x)+\frac{1}{2^k}\|x\|$ for every $x\in \eqs_d$.

 For each $i=1,\ldots,d$ from our inductive assumption there exists
 $g_i\in W'$ with $\supp g_i\subset I_i$ such that
 \[  |f(D_{j_k}(x))|\le
 \frac{1}{w(f)}\alpha_{j_k}(x)+\frac{1}{2^k}\|x\|  \]
 for all $k\in I_i$, with the potential exception for $k\in
 \{k_i,k_i+1\}$ where $k_i+1<\supp g_i$. For the rest of the proof
 suppose that $k_i,k_i+1\in I_i$ are indeed exceptions to the
 above inequality.

 We set
 \[
 g=\frac{1}{w(f)}\big(\sum\limits_{i=1}^d(e_{k_i}^*+e_{k_i+1}^*+g_i)
  +\sum\limits_{k\in I_0}e_k^*\big)\]
  and $g=[k_0+2,+\infty)g'$. The family
  $\{e_{k_i}^*, \;e_{k_i+1}^*,\;
  g_i,\;i=1,\ldots,d\}\cup\{e_k^*:\;k\in I_0\}$ consists of
  successive functionals belonging to $W'$, while its cardinality
  does not exceed $4d$. Thus the functional $g'$ belongs to $W'$
  hence the same holds for the functional $g$.
 We have to check that the functional $g$ satisfies the conclusion
 of the proposition.

 Let $x\in \eqs_d$. For $k<k_0$, as we have observed earlier, we have
 that $|f(D_{j_k}(x))|\le
 \frac{1}{2^k}\|x\|.$ The numbers $k_0,k_0+1$, if belong to $I$,
 are the potential exceptions to the required inequality; observe
 also that $k_0+1<\supp g$.
 Let now $k\in I$ with $k>k_0+1$. We distinguish four  cases.\\
  {\bf Case 1.}\quad $k\in \{k_i,k_i+1\}\subset I_i$ for some $i\in
 \{1,\ldots,d\}$.\\
 Then
 \[ |f(D_{j_k}(x))|=
 \frac{1}{w(f)}|f_i(D_{j_k}(x))|\le
 \frac{1}{w(f)}\|D_{j_k}(x)\|\le
 \frac{1}{w(f)}\alpha_{j_k}(x)=\alpha_{j_k}(x)g(e_k).\]
 {\bf Case 2.}\quad $k\in I_i\setminus\{k_i,k_i+1\}$ for some
 $i\in\{1,\ldots,d\}$.\\
 Then
  \begin{eqnarray*} |f(D_{j_k}(x))| &= &
 \frac{1}{w(f)}|f_i(D_{j_k}(x))|\le
\frac{1}{w(f)}\big(
\alpha_{j_k}(x)g_i(e_k)+\frac{1}{2^k}\|x\|\big)\\
 & \le &
 \alpha_{j_k}(x)g(e_k)+\frac{1}{2^k}\|x\|.
 \end{eqnarray*}
 {\bf Case 3.}\quad $k\in I_0$.\\
 Then, since also $k>k_0+1$ we get that
 \[  |f(D_{j_k}(x))|\le
 \frac{1}{w(f)}\alpha_{j_k}(x)+\frac{1}{2^k}\|x\|=
 \alpha_{j_k}(x)g(e_k)+\frac{1}{2^k}\|x\|.\]
 {\bf Case 4.}\quad $k\in I\setminus \bigcup\limits_{i=0}^dI_i$.\\
 Then $|f(D_{j_k}(x))|=0$.

 The proof of the proposition is complete.
 \end{proof}

 \begin{lemma}\label{Nlem12}
 Let $g\in W'$ and $x\in\eqs_d$. Then
 \[ \sum\limits_{k=1}^{\infty}\alpha_{j_k}(x)|g(e_k)|\le C_1
 \|x\|\]
 where $C_1=\sum\limits_{i=1}^{\infty }\frac{i+1}{m_{2i}}$.
 \end{lemma}
 \begin{proof}[\bf Proof.]
 We set
 \[F_1=\{k:\; \frac{1}{m_2}<|g(e_k)|\}\]
 and for $i=2,3,\ldots$ we set
 \[F_i=\{k:\; \frac{1}{m_{2i}}<|g(e_k)|\le \frac{1}{m_{2i-2}}\}.\]
 Since $m_1=m_2=2$, if $F_1\neq\emptyset$ then $g=\pm e_r^*$ and
 the conclusion  trivially follows (since $C_1\ge 1$).
 Suppose now that $F_1=\emptyset$.
 From  the claim in the proof of Lemma \ref{Nlem5} we get that
 $\#F_i\le (4n_{2i-1})^4\le n_{2i}$ for each $i=2,3,\ldots$.
  We set
 \[L_i=\{ k\in F_i:\; n_{2i}<j_k\}\;mbox{ and }\;G_i=F_i\setminus
 L_i=\{ k\in F_i:\;j_k\le n_{2i}\}.\] Since
 \[\#\{ j_k:k\in L_i\}\le \#L_i\le \#F_i\le n_{2i}<\min\{ j_k:k\in L_i\},\]
  Lemma \ref{Nlem3} yields that
 \[\sum_{k\in L_i}\alpha_{j_k}(x)\le\| x\|.\]
 On the other hand, by Property (iii) of the sequence $(j_k)_{k=1}^{\infty }$,
  we have $n_{2i}<j_i$, and hence, $G_i\subset\{1,\ldots,i-1\}$. Thus, for $i\geq 2$,
 \[\sum_{k\in F_i}\alpha_{j_k}(x)=\sum_{k\in G_i}\alpha_{j_k}(x)+\sum_{k\in L_i}\alpha_{j_k}(x)
 \le (i-1)\|x\|+\|x\|=i\|x\|.\] We conclude that
 \begin{eqnarray*}
 \sum\limits_{k=1}^{\infty}\alpha_{j_k}(x)|g(e_k)|& =
 &\sum\limits_{i=2}^{\infty}\sum\limits_{k\in
 F_i}\alpha_{j_k}(x)|g(e_k)|\\
 & \le  &
 \sum\limits_{i=2}^{\infty}\frac{1}{m_{2i-2}}\big(\sum\limits_{k\in
 F_i}\alpha_{j_k}(x)\big)\\
 & \le
 &\big(\sum\limits_{i=2}^{\infty}\frac{i}{m_{2i-2}}\big)\|x\|=C_1\|x\|.
 \end{eqnarray*}
  \end{proof}

\begin{proof}[\bf Proof of Proposition \ref{Nprop4}.]
 Firstly we shall show the bound of the norm of the
operator  $D=\sum\limits_k\lambda_k D_{j_k}$.
 Let $x\in\eqs_d$. We shall show that for every $f\in K_d$, it
 holds that
 \[ |f(\sum\limits_k\lambda_k D_{j_k}(x))|\le C_0\cdot
 \sup\limits_k|\lambda_k|\cdot \|x\|.\]

 Let $f\in K_d$. From Proposition
  \ref{Nprop3} there exists $g\in W'$ having nonnegative
  coordinates and $k_0\in\N$ such the
  \[  |f(D_{j_k}(x))|\le
  \alpha_{j_k}(x)g(e_k)+\frac{1}{2^k}\|x\|\]
  for all $k\not\in \{k_0,k_0+1\}$.
  Therefore
  \begin{eqnarray*}
  |f(\sum\limits_k\lambda_k D_{j_k}(x))| & \le &
    \sum\limits_k|\lambda_k |\cdot |f(D_{j_k}(x))|\le
    \sup\limits_k|\lambda_k|\cdot \sum\limits_k |f(D_{j_k}(x))|\\
    &\le & \sup\limits_k|\lambda_k|\cdot
    \bigg(
    |f(D_{j_{k_0}}(x))|+|f(D_{j_{k_0+1}}(x))|\\
     & &\quad+\sum\limits_{k\not\in\{k_0,k_0+1\}}
    \big(\alpha_{j_k}(x)g(e_k)+\frac{1}{2^k}\|x\|\big)\bigg)\\
 &\le & \sup\limits_k|\lambda_k|\cdot
    \bigg(\|D_{j_{k_0}}(x)\|+\|D_{j_{k_0+1}}(x)\|
    +\sum\limits_{k=1}^{\infty}\frac{1}{2^k}\|x\|\\
    & &\quad
    +\sum\limits_{k=1}^{\infty}\alpha_{j_k}(x)g(e_k)\bigg)\\
  &\le & \sup\limits_k|\lambda_k|\cdot \bigg(3\|x\|
  +\sum\limits_{k=1}^{\infty}\alpha_{j_k}(x)g(e_k)\bigg).
  \end{eqnarray*}
  From Lemma \ref{Nlem12} we get that
  \begin{equation*}
 \sum\limits_{k=1}^{\infty}\alpha_{j_k}(x)g(e_k)\le C_1\|x\|,
 \end{equation*}
 where $C_1=\sum\limits_{i=1}^{\infty }\frac{i+1}{m_{2i}}$.
 Thus the  operator $D=\sum\limits_k\lambda_k D_{j_k}$ is bounded with $\|D\|\le
 C_0\cdot\sup\limits_k|\lambda_k|$ where $C_0=3+C_1$.

 The fact that the operator $D:\eqs_d\to\eqs_d$ is strictly singular
 follows from the fact
 that $\lim\limits_n D(e_n)=0$ in conjunction to the HI property
 of $\eqs_d$ (see Proposition 1.2 of \cite{AT1}).
 \end{proof}

 \section{The  structure of the space $\mc{L}_{\diag}(\eqs_d)$}\label{Nsec4}

 In this section we
  define the space $J_{T_0}$, which is
 the Jamesification of the space $T_0$ studied in section \ref{Nsec1}.
 We state the finitely block representability of $J_{T_0}$
  in $\eqs_d$ (the proof of this result is presented in the next section)
  and apply it in order to study  the structure of the
  space $\mc{L}_{\diag}(\eqs_d)$ of  diagonal operators
  on $\eqs_d$.
    We start with the definition of
 the space $J_{T_0}$.

 \begin{definition}\label{Ndef12}
 The space $J_{T_0}$ is defined to be the space
  \[
 J_{T_0}=T\big[G,\big(\mc{A}_{n_j},\frac{1}{m_j}\big)_{n\in\N}\big]
 \]
 where $G=\{\pm \chi_I:\;I\text{ finite interval of }\N\}$. This means
 that $J_{T_0}$ is the completion of $(c_{00}(\N),\|\cdot\|_{D_0})$
 where $D_0$ is the
 minimal subset of $c_{00}(\N)$ such that:
 \begin{enumerate}
 \item[(i)] The set $G$ is a subset of $D_0$.
 \item[(ii)] The set $D_0$ is closed in the
 $(\mc{A}_{n_j},\frac{1}{m_j})$ operation for every $j\in\N$.
 \end{enumerate}
 \end{definition}

 \begin{remark}\label{Nrem5}
 An alternative description of the space $J_{T_0}$ is the
 following.
Let $(t_n)_{n\in\N}$ be the standard Hamel basis of $c_{00}(\N)$.
The norm $\|\cdot\|_{J_{T_0}}$ is defined as follows: For every
$x=\sum\limits_{n=1}^{\infty}x(n)t_n\in c_{00}(\N)$ we set
\[\|x\|_{J_{T_0}}=\sup\big \{ \|\sum_{k=1}^l \big ( \sum_{n\in
I_k} x(n)\big )e_k \|_{T_0}, \; l\in \N,\; I_1<I_2<\ldots <I_l
\text{ intervals  of }\N\big \}.\] The space $J_{T_0}$ is the
completion of $(c_{00}(\N), \|\cdot\|_{J_{T_0}}).$
\end{remark}

 \begin{proposition}\label{Nprop11}
 For the space $J_{T_0}$ the following hold.
 \begin{enumerate}
 \item[(i)] The sequence $(t_n)_n$ is a normalized bimonotone
 Schauder   basis of the space  $J_{T_0}$.
 \item[(ii)] For every  $j\in {\mathbb N}$, we have the following estimates:
   \begin{eqnarray*}
   && \|\frac{1}{2p_j}\sum\limits_{k=1}^{p_j}t_{2k-1}\|_{J_{T_0}}=\frac{1}{2}\\
   && \|\frac{1}{2p_j}\sum_{k=1}^{2p_j}(-1)^{k+1} t_{k}\|_{J_{T_0}}
    =\|\frac{1}{2p_j}\sum\limits_{k=1}^{2p_j}e_k\|_{T_0}\le\frac{4}{m_j}.
 \end{eqnarray*}
 In particular the basis $(t_n)_{n\in\N}$ is not unconditional.
 \end{enumerate}
 \end{proposition}
  \begin{proof}[\bf Proof.] The proof that $(t_n)_{n\in\N}$ is a normalized
  bimonotone Schauder
  basis is standard. We set
  $x=\frac{1}{2p_j}\sum\limits_{k=1}^{p_j}t_{2k-1}$. The inequality
  $\|x\|_{J_{T_0}}\le \frac{1}{2}$ is obvious while from the
  action of the functional $f=\chi_I\in D_0$, where
  $I=\{1,2,\ldots,2p_j-1\}$, on the vector $x$ we obtain that $\|x\|_{J_{T_0}}\ge
  \frac{1}{2}$.

  Setting $l=2p_j$ and $I_k=\{k\}$ for $1\le k\le l$ we get
  $\|\sum\limits_{k=1}^{2p_j}(-1)^{k+1}t_k\|_{J_{T_0}}\ge
  \|\sum\limits_{k=1}^{2p_j}(-1)^{k+1}e_k\|_{T_0}=
  \|\sum\limits_{k=1}^{2p_j}e_k\|_{T_0}$
  where  the inequality follows from Remark \ref{Nrem5}, while the
  equality is a consequence of  the 1-unconditionality of the basis
  $(e_k)_{k\in\N}$ of $T_0$ (Remarks \ref{Nrem1} and \ref{Nrem2}).

 Let's explain now the inequality
 $\|\sum\limits_{k=1}^{2p_j}(-1)^{k+1}t_k\|_{J_{T_0}}\le
 \|\sum\limits_{k=1}^{2p_j}e_k\|_{T_0}$. We observe that for
 every interval $I$ of $\N$ the quantity $\sum\limits_{k\in
 I}(-1)^{k+1}$ is either equal to $-1$ or to 0 or to $1$.
 Thus the inequality follows from Remarks \ref{Nrem5} and \ref{Nrem1}.

 Finally the inequality
 $\|\frac{1}{2p_j}\sum\limits_{k=1}^{2p_j}e_k\|_{T_0}\le\frac{4}{m_j}$
 follows from Lemma \ref{Nlem2}.
 \end{proof}

 \begin{theorem}\label{Nth2} There exists a positive constant $c$ such that
 the  basis $(t_n)_{n\in\N}$ of $J_{T_0}$ is $c$ - finitely
 representable in every block subspace of $\eqs_d$. This means that,
 for every block subspace $Z$ of $\eqs_d$ and every
 $N\in\N$, there exists a finite block sequence
 $(z_k)_{k=1}^N$ in $Z$ such that, for every
 choice of scalars $(\mu_k)_{k=1}^N$, we have that
 \[\|\sum\limits_{k=1}^N  \mu_kt_k\|_{J_{T_0}}\le
 \|\sum\limits_{k=1}^N  \mu_kz_k\|_{\eqs_d}\le c\cdot
 \|\sum\limits_{k=1}^N  \mu_kt_k\|_{J_{T_0}}\; .\]
 \end{theorem}
 We shall give the  proof of Theorem \ref{Nth2} in the next
 section.
 Let us note that, since the basis $(t_n)_{n\in\N}$ of
 $J_{T_0}$ is not unconditional, Theorem \ref{Nth2} implies in
 particular that the space $\eqs_d$ does not contain any unconditional basic
 sequence. Of course,  in Theorem \ref{Nth1}, we  have already
 proved the stronger result that the space $\eqs_d$ is
 Hereditarily Indecomposable.

 From Theorem \ref{Nth2} and Proposition \ref{Nprop11} we
 immediately get the following.

 \begin{corollary}\label{Ncor1}
 Let $Z$ be any block subspace of $\eqs_d$ and let $j\in\N$.
 Then there exists a finite block sequence $(y_k)_{k=1}^{2p_j}$ in
 $Z$ such that
 \[
 \|\frac{1}{2p_j}\sum\limits_{k=1}^{p_j}y_{2k-1}\|\ge\frac{1}{2}
 \qquad\text{ and }\qquad
 \|\frac{1}{2p_j}\sum\limits_{k=1}^{2p_j}(-1)^{k+1}y_{k}\|\le
 \frac{4c}{m_j}\]
 \end{corollary}

 \begin{theorem}\label{Nth3}
 There exist bounded strictly singular non-compact diagonal operators on
 the space $\eqs_d$. Especially, given any infinite dimensional subspace $Z$ of
 $\eqs_d$ there exists a bounded strictly singular diagonal operator on
 $\eqs_d$ such that its restriction  on $Z$ is a non-compact operator.

 Moreover the space $\mc{L}_{\diag}(\eqs_d)$ of all bounded diagonal operators
 on the space $\eqs_d$ contains an isomorphic copy of
 $\ell_{\infty}(\N)$.
 \end{theorem}
 \begin{proof}[\bf Proof.]
 By standard perturbation arguments and passing to a subspace we
 may assume that $Z$ is a block subspace of $\eqs_d$.
 We inductively construct vectors
 $(y^j_k)_{k=1}^{2p_j}$, $j=1,2,\ldots$ in $Z$, satisfying the
 conclusion of Corollary \ref{Ncor1} and moreover
 $y^j_{2p_j}<y^{j+1}_1$ for each $j$.

 For $j=1,2\ldots$ and $1\le i\le p_j$ we set
 $I^j_i=\ran(y^j_{2i-1})$ and we define the diagonal operator
 $D_j:\eqs_d\to\eqs_d$ by the rule
 $D_j(x)=\frac{1}{m_j}\sum\limits_{i=1}^{p_j}I^j_{i}x$. We also
 consider for $j=1,2,\ldots$ the vector
 $x_j=\frac{m_j}{2p_j}\sum\limits_{k=1}^{2p_j}(-1)^{k+1}y_k^j$ which belongs to $Z$.
 Then $\|x_j\|\le 4c$, $\|D_jx_j\|\ge \frac{1}{2}$ while
 $D_lx_j=0$ for $l\neq j$.

 Let now $M=\{j_k:\;k\in\N\}$ be any subset of $\N$ satisfying
 conditions (i), (ii), (iii) in the statement of Proposition \ref{Nprop4}.
 Then, from Proposition \ref{Nprop4}, the diagonal operator
 $D=\sum\limits_{k=1}^{\infty}D_{j_k}$ is bounded and strictly
 singular. The restriction of $D$ on $Z$ is non-compact, since the
 block sequence $(x_{j_k})_{k\in\N}$ is bounded, while the
 sequence
 $(Dx_{j_k})_{k\in\N}$ does not have any convergent subsequence.

 For $M=\{j_k:\;k\in\N\}$ as above,  Proposition \ref{Nprop4}
 yields that for every
 $(\lambda_k)_{k\in\N}\in\ell_{\infty}(\N)$, the diagonal operator
 $\sum\limits_{k=1}^{\infty}\lambda_kD_{j_k}$ is bounded with
  $\|\sum\limits_{k=1}^{\infty}\lambda_kD_{j_k}\|\le
 C_0\cdot\sup\limits_k|\lambda_k|$. On the other hand the action of
 the operator $\sum\limits_{k=1}^{\infty}\lambda_kD_{j_k}$ to the
 vector $x_{j_m}$ yields that
 $\|\sum\limits_{k=1}^{\infty}\lambda_kD_{j_k}\|\ge
 \frac{|\lambda_m|\cdot\|D_{j_m}(x_{j_m})\|}{\|x_{j_m}\|}\ge
 \frac{1}{8c}\cdot |\lambda_m|$ for each $m$.
 Hence
 \[  \frac{1}{8c}\cdot \sup\limits_k|\lambda_k|\le
    \|\sum\limits_{k=1}^{\infty}\lambda_kD_{j_k}\|\le
 C_0\cdot\sup\limits_k|\lambda_k|. \]
 The proof of the theorem is complete.
 \end{proof}

 \section{The finite block representability of $J_{T_0}$ in $\eqs_d$}\label{Nsec5}

  The content of this section is  the proof of  Theorem \ref{Nth2}.
 Let $N\in\N$ and let $Z$ be any block subspace
  of $\eqs_d$. We first choose  $j\ge 2$ with $2p_j\ge N$ and $i>j$ such
  that $m_{2i-1}>38p_j$. Then
   we select $(x_r,\phi_r)_{r=1}^{n_{2i+1}}$ a $(6,2i+1)$
  dependent sequence with $x_r\in Z$ and $\min\supp x_1>m_{2i+1}$ (this is
  done with an inductive application of Proposition \ref{Nprop8}).
  The fact that $(\phi_r)_{r=1}^{n_{2i+1}}$ is a special sequence
  yields that the functional
 \[ \Phi=\frac{1}{m_{2i}}(\phi_1+\phi_2+\cdots+\phi_{n_{2i+1}}) .\]
 is a $2i+1$ special functional and thus belong to the norming set
 $K_d$.

   We set
  $M=\frac{n_{2i+1}}{2p_j}$ and observe that $M\ge (4n_{2i})^2$.
 For $1\le k\le 2p_j$ we set
 \[ y_k=\frac{m_{2i}}{M} \sum\limits_{r=(k-1)M+1}^{kM}x_r.\]
 We also consider the functionals
 \[y_k^*=\frac{1}{m_{2i}}\sum\limits_{r=(k-1)M+1}^{kM}\phi_r\]
 for $1\le k\le 2p_j$, and we notice that $y_k^*\in K_d$
 (since each $y_k^*$ is the restriction of $\Phi$ on some interval) with
 $\ran y_k=\ran y_k^*$ and $\|y_k\|\ge y_k^*(y_k)=1$.
 Observe also, that for every subinterval $I$ of
 $\{1,2,\ldots,2p_j\}$, the functionals $\pm\sum\limits_{k\in I}y_k^*$
 also belong to $K_d$.

  Our aim is prove that for every choice of scalars
 $(\mu_k)_{k=1}^{2p_j}$ we have that
   \begin{equation}\label{Neq1}
   \|\sum\limits_{k=1}^{2p_j}\mu_kt_k\|_{J_{T_0}}\le
   \|\sum\limits_{k=1}^{2p_j}\mu_ky_k\|\le
   150\cdot\|\sum\limits_{k=1}^{2p_j}\mu_kt_k\|_{J_{T_0}}.
   \end{equation}
 This will finish the proof of Theorem \ref{Nth2} for
 $c=150$.
 We begin with the  proof of the left side inequality of \eqref{Neq1}
 which is the easy one.
 \begin{proof}[\bf Proof of the left side inequality of
 \eqref{Neq1}]
 It is enough to prove that for every choice of scalars
 $(\mu_k)_{k=1}^{2p_j}$ and every $g\in D_0$ (recall that $D_0$ is
 the norming set of the space $J_{T_0}$; see Definition \ref{Ndef12})
 there exists $f\in K_d$ such that
 $g\big(\sum\limits_{k=1}^{2p_j}\mu_kt_k\big)=f\big(\sum\limits_{k=1}^{2p_j}\mu_ky_k\big)$.

 Let $g\in D_0$. We may assume that $\supp g\subset \{1,2,\ldots
 2p_j\}$. Let $(g_a)_{a\in\mc{A}}$ be a tree of the functional
 $g$. We shall build functionals $(f_a)_{a\in\mc{A}}$ in $K_d$
 such that
 $g_a\big(\sum\limits_{k=1}^{2p_j}\mu_kt_k\big)
 =f_a\big(\sum\limits_{k=1}^{2p_j}\mu_ky_k\big)$
 for each $a\in\mc{A}$. Then the functional $f=f_0$ (where
 $0\in\mc{A}$ is the root of the tree $\mc{A}$) satisfies the
 desired property.

 For $a\in \mc{A}$ which is maximal the functional $g_a$ is of the
 form $g_a=\e \chi_I$ where $\e\in\{-1,1\}$ and
 $I$ is a subinterval of $\{1,2,\ldots 2p_j\}$. We set
 $f_a=\e \sum\limits_{k\in I}y_k^*$ and the desired equality holds
 since $y_k^*(y_k)=1$ for each $k$.
  Let now $a\in\mc{A}$ be non maximal and suppose that the functionals
 $(f_\beta)_{\beta\in S_a}$ have been defined. The functional
 $g_a$ has an expression $g_a=\frac{1}{m_q}\sum\limits_{\beta\in
 S_a}g_\beta$ with $\#S_a\le n_q$, for some $q\in \N$. We set
 $f_a=\frac{1}{m_q}\sum\limits_{\beta\in S_a}f_\beta$. Then $f_a\in K_d$
 while  the required equality is obvious.
  The inductive construction is complete.
 \end{proof}

 Before passing to the proof of the right side inequality of
 \eqref{Neq1} we need some preliminary lemmas.

  \begin{lemma}\label{Nlem9}
  Consider the vector $x=\frac{1}{M}\sum\limits_{l=1}^Me_l$ in the
  auxiliary space $T'$ (recall the the auxiliary space $T'$
  and its norming set $W'$ have been defined in Definition \ref{Ndef7}).
  Then
  \begin{enumerate}
  \item[(i)] If either $f\in W'$ with $w(f)\ge m_{2i+1}$ or
   $f$ is the result of an $(\mc{A}_{4n_{2i+1}},\frac{1}{m_{2i}})$
  operation then \[|f(x)|\le \frac{1}{w(f)}.\]
  \item[(ii)] If either $f\in W'$ with $w(f)<m_{2i}$ or  $f$ is the
  result of an $(\mc{A}_{4n_{2i}},\frac{1}{m_{2i}})$ operation then
  \[|f(x)|\le \frac{2}{w(f)\cdot m_{2i}}.\]
  \end{enumerate}
  \end{lemma}
  \begin{proof}[\bf Proof.]
  Part (i) is obvious. In order to prove part (ii) consider $f\in W'$ such that
  either $w(f)<m_{2i}$ or $f$ is the result of an $(\mc{A}_{4n_{2i}},\frac{1}{m_{2i}})$
  operation. In either case the functional $f$ takes the form
  $f=\frac{1}{w(f)}\sum\limits_{k=1}^df_k$  with $f_1<f_2<\cdots<f_d$ in $W'$ and $d\le 4n_{2i}$.
  We set $D_k=\{l:\; |f_k(e_l)|>\frac{1}{m_{2i}}\}$ for $k=1,2,\ldots ,d$ and
  $D=\bigcup\limits_{k=1}^dD_k$. From the claim in the proof of Lemma \ref{Nlem5}
  we get that $\#(D_k)\le (4n_{2i-1})^4$ for each $k$, thus
  $\#(D)\le 4n_{2i}\cdot(4n_{2i-1})^4$.

   Taking into account that $M\ge (4n_{2i})^2\ge
   4n_{2i} \cdot (4n_{2i-1})^4 \cdot m_{2i}$ we deduce that
   \begin{eqnarray*}
   |f(x)| & \le & |f_{|D}(x)| +|f_{|(\N\setminus D)}(x)|\\
    &\le&       \frac{1}{w(f)}\cdot\frac{1}{M}\cdot \#(D)+\frac{1}{w(f)}\cdot \frac{1}{m_{2i}}\\
    &\le&       \frac{1}{w(f)}\Big(\frac{4n_{2i}\cdot (4n_{2i-1})^4}{M}+ \frac{1}{m_{2i}}\Big)\\
    &\le&  \frac{2}{w(f)\cdot m_{2i}}.
   \end{eqnarray*}

  \end{proof}

  \begin{lemma}\label{Nlem10}
   For $1\le k\le 2p_j$ we have the following.
 \begin{enumerate}
  \item[(i)] If either $f\in K_d$ with $w(f)<m_{2i}$ or  $f$ is the
  result of an $(\mc{A}_{n_{2i}},\frac{1}{m_{2i}})$ operation then
  \[|f(y_k)|\le \frac{54}{w(f)}.\]
  \item[(ii)] If either $f\in K_d$ with $w(f)\ge m_{2i+1}$ or $f$
  is  a $2i+1$ special functional
  (i.e. $f=Eh$ where $h$ is the result of a
  $(\mc{A}_{n_{2i+1}},\frac{1}{m_{2i}})$ operation on an $n_{2i+1}$
  special sequence)
  then
  \[ |f(y_k)|\le \frac{18 m_{2i}}{w(f)}+\frac{36m_{2i}}{M}\le \frac{19}{m_{2i}}.\]
  \end{enumerate}
  In particular $\|y_k\|\le 36$.
  \end{lemma}
 \begin{proof}[\bf Proof.] From Remark \ref{Nrem4} it follows that
 the sequence  $(x_r)_{r\in\N}$ (and thus every subsequence)
 is an $(18,\frac{1}{n_{2i+1}^2})$  R.I.S. The result follows from
 Proposition \ref{Nprop5} and Lemma \ref{Nlem9}.
 \end{proof}

 \begin{proof}[\bf Proof of the left side inequality of
 \eqref{Neq1}]
 Let $f\in K_d$. We fix a tree $(f_a)_{a\in\mc{A}}$ of the
 functional $f$. We set
 \[\mc{B}'=\{a\in\mc{A}:\; f_a\mbox{ is a } 2i+1\mbox{ special
 functional}\}.
   \]
 Let $\beta\in \mc{B}'$. Then the functional $f_\beta$ takes the
 form
 \[ f_\beta=\e_\beta
 \frac{1}{m_{2i}}E(\phi_1+\cdots\phi_{l_0}+\psi_{l_0+1}+\cdots+\psi_{n_{2i+1}})\]
 where $\e_\beta\in\{-1,1\}$, $E$ is an interval of $\N$ and
 $(\phi_1,\ldots,\phi_{l_0},\psi_{l_0+1},\ldots,\psi_{n_{2i+1}})$
 is an $n_{2i+1}$ special sequence with $\psi_{l_0+1}\neq
 \phi_{l_0+1}$. For $\beta$ and $f_\beta$ as above, we set
 \[I_\beta=\big\{k\in\{1,2,\ldots,2p_j\}:\;\supp y_k\subset\ran
 E(\phi_1+\cdots\phi_{l_0})\big\}. \]
 Let
 \[\mc{B}=\{\beta\in \mc{B}':\; I_\beta\neq \emptyset\}. \]
 We notice that
 \begin{enumerate}
 \item[(i)] For every $\beta\in \mc{B}$, the set $I_\beta$ is a
 subinterval of $\{1,2,\ldots,2p_j\}$.
 \item[(ii)] For $\beta_1,\beta_2\in\mc{B}$ with
 $\beta_1\neq\beta_2$ we have that $I_{\beta_1}\cap
 I_{\beta_2}=\emptyset$. In particular $\sum\limits_{\beta\in\mc{B}}\#(I_\beta)\le
 2p_j$.
 \item[(iii)] For every $\beta\in \mc{B}$ we have that
 $f_\beta(\sum\limits_{k\in I_\beta}\mu_ky_k)=\e_\beta\sum\limits_{k\in
 I_\beta}\mu_k$.
 \end{enumerate}
 We set
 \[ F=\bigcup\limits_{\beta\in\mc{B}}I_\beta.\]

 \begin{claim1} We have \quad
 $|f(\sum\limits_{k\in F}\mu_ky_k)|\le
 3\cdot \|\sum\limits_{k=1}^{2p_j}\mu_kt_k\|_{J_{T_0}}.$
 \end{claim1}
 \begin{proof}[\bf Proof of Claim 1.]
 We partition the set $\mc{B}$ into two subsets as follows:
 \begin{eqnarray*}
   \mc{B}_1&=&\{\gamma\in\mc{B}:\; \text{ there exists
 }\beta\in\mc{B}\text{ with }\beta\prec \gamma\} \\
  \mc{B}_2&=& \{\gamma\in\mc{B}:\;\beta\not\in\mc{B} \text{ for every } \beta\prec
 \gamma\}.
 \end{eqnarray*}

 We shall first estimate
 $|f(\sum\limits_{\gamma\in\mc{B}_1}\sum\limits_{k\in
 I_{\gamma}}\mu_ky_k)|$. Let $\gamma\in\mc{B}_1$ and consider
 $\beta\in\mc{B}$ with $\beta\prec\gamma$. The functional
 $f_{\beta}$ is, as we have mentioned before, of the form
 \[ f_\beta=\e_\beta
 \frac{1}{m_{2i}}E(\phi_1+\cdots\phi_{l_0}+\psi_{l_0+1}+\cdots+\psi_{n_{2i+1}})\]
 with $\phi_{l_0+1}\neq \psi_{l_0+1}$.
 Then $\supp f_\gamma\subset \supp \psi_{l}$ for some $l\ge
 l_0+1$. Since $\psi_l$ is not a special functional we obtain that
 $f_\gamma\neq \psi_l$. Thus
 \[|\psi_l(\sum\limits_{k\in I_{\gamma}}\mu_ky_k)|\le
 \frac{1}{w(\psi_l)}|f_\gamma(\sum\limits_{k\in
 I_{\gamma}}\mu_ky_k)|.  \]
 From the definition of special functionals we get that
 $w(\psi_l)>w(\phi_1)>n_{2i+1}^2$. We also have that
 $|f_\gamma(\sum\limits_{k\in I_{\gamma}}\mu_ky_k)|
 =|\sum\limits_{k\in I_{\gamma}}\mu_k|
 \le
 \max\limits_k|\mu_k|\cdot \#(I_\gamma)$. Thus
 \[ |f(\sum\limits_{k\in I_{\gamma}}\mu_ky_k)|\le
 |\psi_l(\sum\limits_{k\in I_{\gamma}}\mu_ky_k)|\le
 \frac{1}{n_{2i+1}^2}\cdot\max\limits_k|\mu_k|\cdot \#(I_\gamma).\]
 We conclude that
 \begin{eqnarray*}
   |f(\sum\limits_{\gamma\in\mc{B}_1}\sum\limits_{k\in
 I_{\gamma}}\mu_ky_k)| &\le &
  \sum\limits_{\gamma\in\mc{B}_1}|f(\sum\limits_{k\in
 I_{\gamma}}\mu_ky_k  )|
      \le  \sum\limits_{\gamma\in\mc{B}_1}\frac{1}{n_{2i+1}^2}\cdot\max\limits_k|\mu_k|\cdot
      \#(I_\gamma)\\
      &\le& \max\limits_k|\mu_k|\cdot
      \frac{2p_j}{n_{2i+1}^2}\le\max\limits_k|\mu_k|
      \le  \|\sum\limits_{k=1}^{2p_j}\mu_k t_k\|_{J_{T_0}}.
 \end{eqnarray*}

 Our next estimate concerns
 $|f(\sum\limits_{\gamma\in\mc{B}_2}\sum\limits_{k\in
 I_{\gamma}}\mu_ky_k)|$. From the definition of $\mc{B}_2$, its
 elements are incomparable nodes of the tree $\mc{A}$. We consider
 the minimal complete subtree $\mc{A}'$ of $\mc{A}$ containing
 $\mc{B}_2$, i.e.
 \[  \mc{A}'=\{a\in\mc{A}:\; \text{ there exists }\beta\in
 \mc{B}_2\text{ with }a\preceq\beta\}. \]
 For every $a\in\mc{A}'$ we set
 \[R_a=\bigcup\limits_{\beta\in\mc{B}_2,\;\beta\succeq a}I_\beta.\]
 As follows from the definition of the sets $I_\beta$, for every
 non maximal $a\in\mc{A}'$, the sets $(R_\beta)_{\beta\in S_a\cap
 \mc{A}'}$ are pairwise disjoint.

 For every $a\in\mc{A}'$ we shall  construct
 functionals $g_a,h_a\in c_{00}(\N)$ such that the following conditions are
 satisfied:
 \begin{enumerate}
 \item[(i)] $\supp g_a\subset R_a$ and $\supp h_a\subset R_a$.
 \item[(ii)] $g_a\in D_0$ (the norming set $D_0$ of the space $J_{T_0}$
 has been defined in Definition \ref{Ndef12})
   and $\|h_a\|_{\infty}\le
 \frac{1}{m_{2i+1}}$.
 \item[(iii)] $g_a(\sum\limits_{k\in R_a}\mu_kt_k)\ge 0$
 and $h_a(\sum\limits_{k\in R_a}\mu_kt_k)\ge 0$.
 \item[(iv)] $|f_a(\sum\limits_{k\in R_a}\mu_ky_k)|\le (g_a+h_a)(\sum\limits_{k\in
 R_a}\mu_kt_k)$.
 \end{enumerate}

 The construction is inductive starting of course with the maximal
 elements of $\mc{A}'$, i.e. with the elements of $\mc{B}_2$.\\
   \smallskip
 \noindent
 {\bf $1\stackrel{st}{=}$ inductive step}\\
 \smallskip
 \noindent    Let $\beta\in \mc{B}_2$. Then $f_\beta$ is a $2i+1$ special
  functional, $R_\beta= I_\beta$ and
  $|f_\beta(\sum\limits_{k\in R_\beta}\mu_ky_k)|=|\sum\limits_{k\in R_\beta}\mu_k|.$
  We set $\e=\sgn(\sum\limits_{k\in R_\beta}\mu_k)$,
  $g_\beta=\e\cdot\chi_{I_{\beta}}$ and $h_a=0$. It is clear that
   our requirements about $g_\beta,h_\beta$ are satisfied.\\
  \smallskip
 \noindent
 {\bf General inductive step}\\
  \smallskip
 \noindent Let $a\in\mc{A}'$, $a\not\in\mc{B}_2$ and assume that
 for every $\gamma\in S_a\cap\mc{A}'$ the functionals $g_\gamma,
 h_\gamma$ have been defined satisfying the inductive assumptions.
 We distinguish three cases.\\
  \smallskip
 \noindent
 {\bf Case 1. } $f_a$ is not a special functional.\\
  \smallskip
 \noindent
 Let $f_a=\frac{1}{m_p}\sum\limits_{\gamma\in S_a}f_{\gamma}$ with
 $\#S_a\le n_p$. We set
 \[ g_a=\frac{1}{m_p}\sum\limits_{\gamma\in
 S_a\cap\mc{A}'}g_{\beta}\qquad\text{ and }\qquad
  h_a=\frac{1}{m_p}\sum\limits_{\gamma\in
 S_a\cap\mc{A}'}h_{\beta}.\]
 Conditions (i), (ii), (iii) are obviously satisfied, while,
  since $R_a=\bigcup\limits_{\gamma\in S_a\cap\mc{A}'}R_{\gamma}$, we get that
 \begin{eqnarray*}
 |f_a(\sum\limits_{k\in R_a}\mu_ky_k)| &= &
 \big| \frac{1}{m_p} \sum\limits_{\gamma\in S_a\cap\mc{A}'}
  f_{\gamma}(\sum\limits_{k\in R_a}\mu_ky_k)         \big|\\
  &\le &  \frac{1}{m_p} \sum\limits_{\gamma\in S_a\cap\mc{A}'}
 |f_{\gamma}(\sum\limits_{k\in R_\gamma}\mu_ky_k)|\\
 &\le &
 \frac{1}{m_p} \sum\limits_{\gamma\in S_a\cap\mc{A}'}
 (g_\gamma+h_\gamma)(\sum\limits_{k\in R_\gamma}\mu_kt_k)\\
 & = & (g_a+h_a)(\sum\limits_{k\in R_a}\mu_kt_k).
 \end{eqnarray*}
  \smallskip
 \noindent
 {\bf Case 2. } $f_a$ is a $2q+1$ special functional for some $q\ge i$.\\
  \smallskip
 \noindent
 Then  $f_a=\e_a
 \frac{1}{m_{2q}}E(\phi_1+\cdots\phi_{l_0}+\psi_{l_0+1}+\cdots+\psi_{n_{2q+1}})$,
 with $\phi_{l_0+1}\neq\psi_{l_0+1}$ (functionals of the form $\phi_r$
 in the expression above may
 appear only if $q=i$; if $q>i$ then $l_0=0$).
 If $q>i$ then $a\not\in\mc{B}'$, hence it has no sense to talk
 about $I_a$. In the case $q=i$
 from the definition of the
 set $\mc{B}_2$ we get that $I_a=\emptyset$.
 Similarly to the proof  concerning $\mc{B}_1$, we obtain that for every
 $\beta\in\mc{B}_2$ with $a\prec\beta$ there exists $l\ge l_0+1$
 such that $\supp f_\beta\subset \supp \psi_l$ and
 \[ |f_a(\sum\limits_{k\in I_\beta}\mu_ky_k)|\le
 \frac{1}{n_{2q+1}^2}\cdot\max\limits_{k\in I_{\beta}} |\mu_k|\cdot \#(I_\beta).\]
 Therefore
 \begin{eqnarray*}
 |f_a(\sum\limits_{k\in R_a}\mu_ky_k)|& \le&
 \sum\limits_{\beta\in\mc{B}_2,\;\beta\succ a}
  |f_a(\sum\limits_{k\in I_{\beta}}\mu_ky_k)| \le
 \frac{1}{n_{2q+1}^2}\cdot\max\limits_{k\in R_{a}}|\mu_k|
 \cdot\sum\limits_{\beta\in\mc{B}_2,\;\beta\succ a} \#(I_\beta) \\
  & \le &  \frac{2p_j}{n_{2q+1}^2}\cdot\max\limits_{k\in R_{a}}|\mu_k|\le
  \frac{1}{n_{2q+1}}\cdot\max\limits_{k\in R_{a}}|\mu_k|.
  \end{eqnarray*}
  We select $k_a\in R_a$ such that $|\mu_{k_a}|=\max\limits_{k\in
  R_a}|\mu_k|$ and we set
  \[g_a=0 \qquad\text{  and }\qquad h_a=\sgn(\mu_{k_a})\cdot\frac{1}{n_{2q+1}}\cdot
 t_{k_a}^*.\]
   \smallskip
 \noindent
 {\bf Case 3. } $f_a$ is a $2q+1$ special functional for some $q< i$.\\
  \smallskip
 \noindent
 Then $f_a$ takes the form
$f_a=\e_a \frac{1}{m_{2q}}E(f_{\gamma_1}+\cdots+f_{\gamma_d})$
with  $d\le n_{2q+1}$.
 Similarly to the proof concerning $\beta\in\mc{B}_1$,
 for every $\beta\in\mc{B}_2$ with $a\prec\beta$ there exists $s$
 such that $\supp f_{\beta}\subset \supp f_{\gamma_s}$, while
 \[ |f_{\gamma_s}(\sum\limits_{k\in
 I_{\beta}}\mu_ky_k)|\le\frac{1}{w(f_{\gamma_s})}\cdot|f_{\beta}(\sum\limits_{k\in
 I_{\beta}}\mu_ky_k)|.
  \]
 Let $s_0$ be such that
 $w(f_{\gamma_{s_0}})<m_{2i+1}<w(f_{\gamma_{s_0+1}})$.
 From the definition of the special sequences and the coding function $\sigma$, we get that
 \[ \#\bigg( \bigcup\limits_{s=1}^{s_0-1}\ran f_{\gamma_s}   \bigg)
  \le \max\supp f_{\gamma_{s_0-1}}<w(f_{\gamma_{s_0}})<m_{2i+1}.\]
 Since for each $k$ we have that $\#\supp y_k\ge M>m_{2i+1}$, it
 follows that for every $s<s_0$ there is no $\beta\in\mc{B}_2$
 such that $\supp f_{\beta}\subset \supp f_{\gamma_s}$.

 If $s>s_0$ and $\beta\in\mc{B}_2$ are such that $\supp
 f_\beta\subset \supp f_{\gamma_s}$ then
 \[ |f_{\gamma_s}(\sum\limits_{k\in
 I_{\beta}}\mu_ky_k)|\le\frac{1}{w(f_{\gamma_s})}|f_{\beta}(\sum\limits_{k\in
 I_{\beta}}\mu_ky_k)|\le
 \frac{1}{m_{2i+2}}\cdot\max\limits_{k\in
 I_\beta}|\mu_k|\cdot\#(I_{\beta}).
 \]
 We select $k_a\in \bigcup\limits_{s>s_0}R_{\gamma_s}$ such that
 $|\mu_{k_a}|=\max\{|\mu_k|:\;k\in
 \bigcup\limits_{s>s_0}R_{\gamma_s}\}$.

 If there is no $\beta\in\mc{B}_2$ such that
 $\gamma_{s_0}\prec\beta$ then we set
 \[  g_a=0 \qquad\text{ and }\qquad h_a=\sgn(\mu_{k_a})\cdot\frac{1}{m_{2i+1}}\cdot t_{k_a}^*.\]
  If there exists $\beta\in\mc{B}_2$ such that
 $\gamma_{s_0}\prec\beta$ then the functionals $g_{s_0}$ and
 $h_{s_0}$ have been defined in the previous inductive step. We
 set
 \[  g_a=g_{s_0} \qquad\text{ and }\qquad h_a=h_{s_0}+\sgn(\mu_{k_a})
              \cdot\frac{1}{m_{2i+1}}\cdot t_{k_a}^*.\]
 Conditions (i), (ii), (iii) are easily established; we shall show condition
 (iv). We assume that  there exists $\beta\in\mc{B}_2$ such that
 $\gamma_{s_0}\prec\beta$ (the modifications are obvious is no
 such $\beta$ exists).
 \begin{eqnarray*}
    |f_a(\sum\limits_{k\in R_a}\mu_ky_k)|&\le  &
    |f_{\gamma_{s_0}}(\sum\limits_{k\in
    R_{\gamma_{s_0}}}\mu_ky_k)|
    +\sum\limits_{s>s_0}|f_{\gamma_s}(\sum\limits_{k\in R_{\gamma_s}}\mu_ky_k) | \\
    & \le &(g_{s_0}+h_{s_0})(\sum\limits_{k\in
    R_{\gamma_{s_0}}}\mu_k t_k)  \\
    && \qquad\qquad\quad + \frac{1}{m_{2i+2}}\cdot\max\limits_{k\in
     \bigcup\limits_{s>s_0}R_{\gamma_s}}|\mu_k|\cdot\sum\limits_{s>s_0}\#(R_{\gamma_s}) \\
      & \le & g_{s_0}(\sum\limits_{k\in R_{\gamma_{s_0}}}\mu_k t_k)+h_{s_0}(\sum\limits_{k\in
    R_{\gamma_{s_0}}}\mu_k t_k)+
    \frac{1}{m_{2i+1}^2}\cdot |\mu_{k_a}|\cdot 2p_j
    \\
     &\le  & g_{s_0}(\sum\limits_{k\in R_{a}}\mu_k t_k)+h_{s_0}(\sum\limits_{k\in
    R_{a}}\mu_k t_k)\\
    & &\qquad\qquad+\sgn(\mu_{k_a})\cdot\frac{1}{m_{2i+1}} \cdot e_{k_a}^*(\sum\limits_{k\in
    R_{a}}\mu_k t_k) \\
    & =&  (g_a+h_a)(\sum\limits_{k\in R_{a}}\mu_k t_k).
 \end{eqnarray*}
 The inductive construction is complete.

 For the the functionals $g_0,h_0$ corresponding
 to the root $0\in\mc{A}$ of the tree $\mc{A}$,
 noticing that
 $R_0=\bigcup\limits_{\beta\in\mc{B}_2}I_\beta$,
   we get that
 \begin{eqnarray*}
 |f(\sum\limits_{\beta\in\mc{B}_2}\sum\limits_{k\in
 I_\beta}\mu_ky_k)|& = &
 |f(\sum\limits_{k\in R_0}\mu_ky_k)|\le (g_0+h_0)(\sum\limits_{k\in R_0}\mu_k t_k)  \\
 &\le &   g_0(\sum\limits_{k\in R_0}\mu_k t_k)+ \frac{1}{m_{2i+1}}\cdot\max\limits_{k\in
 R_0}|\mu_k|\cdot \#(R_0) \\
 & \le & g_0(\sum\limits_{k=1}^{2p_j}\mu_k
 t_k)+ \frac{2p_j}{m_{2i+1}}\cdot
 \max\limits_{1\le k\le 2p_j}|\mu_k|
 \\
 & \le & \|\sum\limits_{k=1}^{2p_j}\mu_k t_k\|_{J_{T_0}}+\max\limits_{1\le k\le
 2p_j}|\mu_k|\\
 &
 \le & 2 \cdot \|\sum\limits_{k=1}^{2p_j}\mu_k t_k\|_{J_{T_0}}.
 \end{eqnarray*}

 Therefore we get that
 \begin{eqnarray*}
 |f(\sum\limits_{k\in F}\mu_ky_k)|& \le &
 |f(\sum\limits_{\gamma\in\mc{B}_1}\sum\limits_{k\in
 I_\gamma}\mu_ky_k)|+ |f(\sum\limits_{\beta\in\mc{B}_2}\sum\limits_{k\in
 I_\beta}\mu_ky_k)|  \\
 & \le & \|\sum\limits_{k=1}^{2p_j}\mu_k
t_k\|_{J_{T_0}}+2 \cdot \|\sum\limits_{k=1}^{2p_j}\mu_k
t_k\|_{J_{T_0}}=3\cdot \|\sum\limits_{k=1}^{2p_j}\mu_k
t_k\|_{J_{T_0}}
 \end{eqnarray*}
 and this finishes the proof of Claim 1.
 \end{proof}

 Next we shall estimate $|f(\sum\limits_{k\not\in F}\mu_ky_k)|$.
 We clearly may restrict our intention to $k\in D$ where
 \[ D=\big\{k\in\{1,2,\ldots,2p_j\}:\;k\not\in F\text{ and } \supp
 f\cap \supp y_k\neq \emptyset\big\}.  \]
  In order to estimate $f(\sum\limits_{k\in D}\mu_ky_k)$ we shall
  split the vector $y_k$, for each $k\in D$, into two parts, the
  initial part $y_k'$ and the final part $y_k''$. The way of the
  split depends on the specific analysis $(f_a)_{a\in\mc{A}}$ of the
  functional $f$ that we have fixed.

  \begin{definition}\label{Ndef13}
  For $k\in D$ and $a\in\mc{A}$ we say that $f_a$ covers $y_k$ if
  \[\supp(f_a)\cap\supp(y_k)=\supp(f)\cap\supp(y_k).\]
  \end{definition}
 Next we introduce some notation which will be used in the rest of the proof.
 \begin{notation}\label{Nnot1}
   We correspond to each $y_k$, for $k\in D$, two vectors $y_k'$, $y_k''$
  defined as follows.\\
    \smallskip
 \noindent
  {\bf Case 1.}\quad  $\#\bigg(\supp(f)\cap\supp(y_k)\bigg)=1$.\\
   \smallskip
 \noindent
   Then there exists a unique maximal node $a_k\in\mc{A}$ such
  $f_{a_k}=e_{l_k}^*$ covers $y_k$. In this case we set $y_k'=y_k$ and $y_k''=0$.\\
   \smallskip
 \noindent
  {\bf Case 2.}\quad  $\#\bigg(\supp(f)\cap\supp(y_k)\bigg)\ge 2$.\\
   \smallskip
 \noindent
   Then there exists a unique node $a_k\in\mc{A}$ such that
  $f_{a_k}$ covers $y_k$ but for every $\beta\in S_{a_k}$,
 $f_{\beta}$ does not cover $y_k$. Let
  \[  \{\beta\in  S_a:\;
  \supp(f_\beta)\cap\supp(y_k)\neq\emptyset\}=\{\beta_1,\beta_2,\ldots,\beta_d\}
  \]
  with $f_{\beta_1}<f_{\beta_2}<\cdots<f_{\beta_d}$.
  We set
  \[ y_k'=y_k|_{[1,\max\supp f_{\beta_1}]}\qquad\text{ and }\qquad
  y_k''=y_k-y_k'.\]
  \end{notation}

  \begin{remark}\label{Nrem7}
    The estimates given in Lemma \ref{Nlem10} for the vectors $y_k$,
    $1\le k\le 2p_j$, remain valid if we replace, for each
   $k\in D$, the  vector   $y_k$  by either  the  vector $y_k'$ or
   by the vector $y_k''$.
  \end{remark}
  The analogue of Definition \ref{Ndef13},
  concerning the vectors $y_k'$, $y_k''$   is the following.
  \begin{definition}\label{Ndef15}
  For $k\in D$ and $a\in\mc{A}$ we say that $f_a$ covers $y_k'$ if
  $\supp(f_a)\cap\supp(y_k')=\supp(f)\cap\supp(y_k')$  while we
  say that $f_a$ covers $y_k''$ if
  $\supp(f_a)\cap\supp(y_k'')=\supp(f)\cap\supp(y_k'')$.
  \end{definition}

 The property of the sequences $(y_k')_{k\in D}$ and
 $(y_k'')_{k\in  D}$ which will play a key role in our proof is
 described in  the following remark.

 \begin{remark}\label{Nrem6}
 \begin{enumerate}
 \item[(i)]  Suppose that $k\in D$ and $a\in\mc{A}$ is a non maximal node
  such that $f_a$ covers $y_k'$ but for every $\beta\in S_a$,
  $f_{\beta}$ does not cover $y_k'$. Then there exists a node
  $\beta_k\in S_a$ (not necessarily unique) such that
  \[ \supp(f_{\beta_k})\cap \supp(y_k')\neq \emptyset  \]
  and
  \[   \supp(f_{\beta_k})\cap \supp(y_l')= \emptyset
  \text{\;\; for all } l\in D \text{\ with }l\neq k.\]
 \item[(ii)] The statement of (i) remains valid if we replace the
  sequence $(y_l')_{l\in D}$ with the sequence $(y_l'')_{l\in D}$.
 \end{enumerate}
   \end{remark}

 \begin{claim2} We have that
 \begin{enumerate}
 \item[(a)]\quad
 $|f(\sum\limits_{k\in D}\mu_ky_k')|\le
 73\cdot \|\sum\limits_{k=1}^{2p_j}\mu_kt_k\|_{J_{T_0}}.$
  \item[(b)] \quad
 $|f(\sum\limits_{k\in D}\mu_ky_k'')|\le
 73\cdot \|\sum\limits_{k=1}^{2p_j}\mu_kt_k\|_{J_{T_0}}.$
 \end{enumerate}
 \end{claim2}
 \begin{proof}[\bf Proof of Claim 2.]
 We shall only show (a). The proof of (b) is almost identical; only minor
 modifications are required.

 For each $a\in \mc{A}$ we set
 \[   D_a=\{k\in D:\; f_a \text{ covers }y_k'\}.  \]
 Setting $\mc{A}'=\{a\in \mc{A}:\; D_a\neq\emptyset\}$, we observe
 that $\mc{A}'$ is a complete subtree of the tree $\mc{A}$.
 We shall construct two  families of functionals
 $(g_a)_{a\in\mc{A}'}$ and $(h_a)_{a\in\mc{A}'}$ such that the
 following conditions are satisfied for every $a\in\mc{A}'$.
 \begin{enumerate}
 \item[(i)] $\supp g_a\subset D_a$ and  $\supp h_a\subset D_a$,
  while $\supp g_a\cap\supp h_a=\emptyset$.
 \item[(ii)] $g_a\in D_0$ and $\|h_a\|_{\infty}\le\frac{1}{m_{2i-1}}$.
 \item[(iii)] $g_a(\sum\limits_{k\in D_a}\mu_kt_k)\ge 0$ and
 $h_a(\sum\limits_{k\in D_a}\mu_kt_k)\ge 0$.
 \item[(iv)]
  $|f(\sum\limits_{k\in D_a}\mu_ky_k')|\le
 (72g_a+h_a)(\sum\limits_{k\in D_a}\mu_kt_k)$.
 \end{enumerate}

 For $a\in\mc{A}'$ which is  non maximal in
 $\mc{A}'$, we set $S_a'=S_a\cap \mc{A}'=\{\beta\in
 S_a:\;D_\beta\neq\emptyset\}$. Observe for later use, that the sets
 $(D_{\beta})_{\beta\in S_a'}$ are successive and pairwise
 disjoint.

 The construction of $(g_a)_{a\in\mc{A}'}$ and $(h_a)_{a\in\mc{A}'}$
 is inductive. Let $a\in\mc{A}'$ and suppose that for every
 $\beta\in \mc{A}'$, $\beta\succ a$ the functionals $g_{\beta}$, $h_{\beta}$ have
 been defined satisfying conditions (i), (ii), (iii), (iv).
 We distinguish the following cases.\\
   \smallskip
 \noindent
  {\bf Case 1.}\quad  $a$ is
 a maximal node of the tree $\mc{A}$. \\
   \smallskip
 \noindent Then $f_a$ is of the form $f_a=e_{l_a}^*$, while
 the set $D_a$ is a singleton,   $D_a=\{k_a\}$. We set
 $g_a=\sgn(\mu_{k_a})\cdot
  t_{k_a}^*$ and $h_a=0$.
  Conditions (i), (ii), (iii) are obvious, while
  from   Remark   \ref{Nrem7} and Lemma \ref{Nlem10} we get that
 \[  |f_a(\sum\limits_{k\in
 D_a}\mu_ky_k')|
 =|\mu_{k_a}|\cdot|f_a(y_{k_a}')|
 \le |\mu_{k_a}| \cdot\|y_{k_a}'\|
 \le 36 \cdot |\mu_{k_a}|\le (72g_a+h_a)(
 \sum\limits_{k\in
 D_a}\mu_kt_k). \]
    \smallskip
 \noindent
  {\bf Case 2.}\quad  $w(f_a)\ge m_{2i+1}$. \\
   \smallskip
 \noindent Then
 from   Remark   \ref{Nrem7} and Lemma \ref{Nlem10} we get that
 $|f_a(y_k')|\le \frac{19}{m_{2i}}$ for every $k\in D_a$, thus,
 taking into account that $\#(D_a)\le 2p_j$ and that from our choice
 of $i$, $38p_j<m_{2i-1}$,
  it follows that
 \[ |f_a(\sum\limits_{k\in D_a}\mu_ky_k')|\le
 \max\limits_{k\in D_a}|\mu_k|\cdot \frac{19\cdot\#(D_a)}{m_{2i}}
 \le \max\limits_{k\in D_a}|\mu_k|\cdot \frac{1}{m_{2i-1}}.\]
  We select $k_a\in D_a$ with $|\mu_{k_a}|=\max\limits_{k\in
 D_a}|\mu_k|$ and we set $g_a=0$ and $h_a=\sgn(\mu_{k_a})\cdot\frac{1}{m_{2i-1}}\cdot
 e_{k_a}^*$.\\
     \smallskip
 \noindent
  {\bf Case 3.}\quad  $f_a$ is the result of an $(\mc{A}_{n_p},\frac{1}{m_p})$
  operation for some $p\le 2i$. \\
   \smallskip
 \noindent Let $f_a=\frac{1}{m_p}\sum\limits_{\beta\in
 S_a}f_\beta$ with $\#S_a\le n_p$. We set
 $T_a=D_a\setminus\bigcup\limits_{\beta\in S_a'}D_{\beta}$.

 From Remark \ref{Nrem6}, for each $k\in T_a$ there exists
 $\beta_k\in S_a$ such that
 $\supp(f_{\beta_k})\cap \supp(y_k')\neq \emptyset$
  and
 $\supp(f_{\beta_k})\cap \supp(y_l')= \emptyset$
 for every $l\in D$, $l\neq k$. This implies that $\beta_k\in
 S_a\setminus S_a'$.
 Since clearly the correspondence
 \begin{eqnarray*}
   T_a & \longrightarrow & S_a\setminus S_a'\\
    k  & \longmapsto   & \beta_k
 \end{eqnarray*}
 is one to one,
 it follows that $\#T_a+\#S_a'\le \#S_a\le n_p$.
 We set
 \[ g_a=\frac{1}{m_p}\big(\sum\limits_{\beta\in S_a'}g_{\beta}+\sum\limits_{k\in
 T_a}\sgn(\mu_k)t_k^*\big)\quad\mbox{ and }\quad
  h_a=\frac{1}{m_p}\sum\limits_{\beta\in S_a'}h_{\beta}.\]
  From our last observation and the inductive assumptions
  it follows that $g_a\in D_0$, while, again from our inductive
  assumptions, we have that
  $\|h_a\|_{\infty}\le \frac{1}{m_{2i-1}}$
  and $g_a(\sum\limits_{k\in D_a}\mu_kt_k)\ge 0$,
 $h_a(\sum\limits_{k\in D_a}\mu_kt_k)\ge 0$.

 For every $k\in T_a$, Remark   \ref{Nrem7} and Lemma \ref{Nlem10}
 yield that $|f_a(y_k')|\le \frac{54}{m_p}\le\frac{72}{m_p}$. Therefore
 \begin{eqnarray*}
    |f_a(\sum\limits_{k\in D_a}\mu_ky_k')| &\le  &
    \sum\limits_{\beta\in S_a}\frac{1}{m_p}|f_{\beta}(\sum\limits_{k\in D_\beta}\mu_ky_k')|
     +\sum\limits_{k\in T_a}|f_a(\mu_ky_k')| \\
  & \le&    \sum\limits_{\beta\in S_a'}
  \frac{1}{m_p}(72g_\beta+h_\beta)(\sum\limits_{k\in D_{\beta}}\mu_kt_k)
   + \sum\limits_{k\in T_a}|\mu_k|\cdot \frac{72}{m_p}   \\
  & =&    (72 g_a+h_a)(\sum\limits_{k\in D_a} \mu_k t_k).
  \end{eqnarray*}
     \smallskip
 \noindent
  {\bf Case 4.}\quad  $f_a$ is a $2i+1$ special functional. \\
   \smallskip
 \noindent Let  $f_a=\e_a
 \frac{1}{m_{2i}}E(\phi_1+\cdots\phi_{l_0}+\psi_{l_0+1}+\cdots+\psi_{d})$,
 where  $\phi_{l_0+1}\neq\psi_{l_0+1}$, $d\le n_{2i+1}$ and $\max E=\max\supp \psi_d$.
  From the definition of the
 sets $F=\bigcup\limits_{\beta\in \mc{B}}I_\beta$ and
 $D=\{k:\;k\not\in F\mbox{ and }\supp
 (f)\cap\supp(y_k)\neq\emptyset\}$ we get that the set
 \[R=\{k\in D_a:\; f_a\mbox{ covers }y_k'\mbox{ and } \supp
 E(\phi_1+\cdots+\phi_{l_0})\cap \supp (y_k')\neq\emptyset\} \]
 contains at most two elements (i.e. $\#R\le 2$).
  We set \[ g_a=\frac{1}{2}\sum\limits_{k\in R}\sgn(\mu_k)t_k^* .\]
  We observe that $|f_a(\sum\limits_{k\in R}\mu_ky_k')|\le
  36\sum\limits_{k\in R}|\mu_k|=72 g_a(\sum\limits_{k\in
  R}\mu_kt_k)$.

  Since $y_k=\frac{m_{2i}}{M}\sum\limits_{r=(k-1)M+1}^{kM}x_r$, the
  vector $y_k'$ takes the form\\
  $y_k'=\frac{m_{2i}}{M}(x_{(k-1)M+1}+\cdots+x_{s-1}+x_s')$ for
  some $s\le kM$ where $x_s'$ is of the form $x_s'=[\min\supp
  x_s,m]x_s$.

  Let $k\in D_a\setminus R$. In order to give an upper
   estimate of the action of $f_a$ on $y_k'$,
   we may assume, without loss of generality,
    that $x_s'=x_s$. Since
   $(x_r,\phi_r)_{r=1}^{n_{2i+1}}$ is a $(6,2i+1)$ dependent
   sequence we have that $w(\psi_l)\neq w(\phi_r)$ for all pairs $(l,r)$
 with $(l,r)\neq (l_0+1,l_0+1)$,
 while $|\psi_{l_0+1}(x_{l_0+1})|\le \|x_{l_0+1}\|\le 6$.
 It follows that
    \begin{eqnarray*}
     |f_a(y_k')| & \le &
     \frac{m_{2i}}{M}\big|\big(\sum\limits_{l=l_0+1}^d\psi_l\big)(\sum\limits_{r=(k-1)M+1}^sx_r)\big|\\
     &\le&\frac{m_{2i}}{M}\sum\limits_{r=(k-1)M+1}^s\sum\limits_{l=l_0+1}^d|\psi_l(x_r)|\\
   & \le & \frac{m_{2i}}{M}\bigg( 6+\sum\limits_{r=(k-1)M+1}^s
   \big(   \sum\limits_{w(\psi_l)<w(\phi_r)}\frac{18}{w(\psi_l)}+
   \sum\limits_{w(\psi_l)>w(\phi_r)}\frac{6}{w(\phi_r)^2}\big)\bigg)\\
  & \le &  \frac{1}{m_{2i+1}}.
  \end{eqnarray*}

 Thus
  \[|f_a(\sum\limits_{k\in D_a\setminus  R}\mu_ky_k')|\le
  \max\limits_{k\in D_a\setminus R}|\mu_k|\cdot
  2p_j\cdot\frac{1}{m_{2i+1}} \le   \max\limits_{k\in D_a\setminus
  R}|\mu_k|\cdot\frac{1}{m_{2i}}.
   \]
 We select $k_a\in D_a\setminus R$ such that
 $|\mu_{k_a}|=\max\limits_{k\in D_a\setminus R}|\mu_k|$ and we set
 $h_a=\sgn(\mu_{k_a})\cdot\frac{1}{m_{2i}}\cdot t_{k_a}^*$.

   We easily get that
    \[
    |f_a(\sum\limits_{k\in D_a}\mu_ky_k')| \le   (72g_a+h_a)(\sum\limits_{k\in
    D_a}\mu_kt_k)\]
  while  inductive assumptions (i), (ii), (iii) are also satisfied
  for the functionals $g_a$, $h_a$.\\
      \smallskip
 \noindent
  {\bf Case 5.}\quad  $f_a$ is a $2q+1$ special functional for some $q<i$. \\
   \smallskip
 \noindent Let   $f_a=\e_a
 \frac{1}{m_{2q}}(f_{\beta_1}+f_{\beta_2}+\cdots+f_{\beta_d})$,
 where  $d\le n_{2q+1}(\le n_{2i-1})$. We set
  \[l_0=\min\{l:\; \max\supp f_{\beta_l}\ge \min\supp y_1' \}.\]
  Then using our assumption that $\min\supp x_1>m_{2i+1}$ (see the choice
  of the dependent sequence $(x_r,\phi_r)_{r=1}^{n_{2i+1}}$ in the beginning
  of the present section), the
  fact that $\min\supp y_1'=\min\supp x_1$ and the definition
  of the special sequences, we get that
  \[  m_{2i+1}<\min\supp y_1' \le \max\supp
  f_{\beta_{l_0}}<w(f_{\beta_{l_0+1}}).\]
 Thus for every $k$, using Lemma \ref{Nlem10}(ii) and Remark \ref{Nrem7}, we get that
  \begin{eqnarray*}
   \sum\limits_{l=l_0+1}^{d}|f_{\beta_l}(y_k')|&\le
     &\sum\limits_{l=l_0+1}^{d}\big(\frac{18m_{2i}}{w(f_{\beta_l})}+\frac{36m_{2i}}{M}\big) \\
  & \le &18m_{2i}\cdot\frac{2}{m_{2i+2}}+n_{2i-1}\cdot \frac{36m_{2i}}{M} \\
  & \le& \frac{2}{m_{2i}}.
  \end{eqnarray*}
  This yields that $\sum\limits_{k\in
  D_a}\sum\limits_{l=l_0+1}^{d}|f_{\beta_l}(y_k')|\le 2p_j\cdot \frac{2}{m_{2i}}\le
  \frac{1}{m_{2i-1}}$.

 We observe that there exists at most one $k_0\in D_a\setminus D_{\beta_{l_0}}$
 such that $\supp f_{\beta_{l_0}}\cap \supp y_{k_0}'\neq
 \emptyset$. Without
 loss of generality, we assume that such a $k_0$ exists.
  We set \[g_a=\frac{1}{2}(g_{\beta_{l_0}}+e_{k_0}^*).\]
 We select $k_a\in D_a\setminus (D_{\beta_{l_0}}\cup\{k_0\})$ such
 that $|\mu_{k_a}|=\max\{|\mu_k|:\;k\in D_a\setminus
 (D_{\beta_{l_0}}\cup\{k_0\})\}$ and we set
 \[ h_a=\frac{1}{m_{2q}}h_{\beta_{l_0}}+ \sgn(\mu_{k_a})\cdot
 \frac{1}{m_{2i-1}}\cdot t_{k_a}^*.\]

  Then conditions (i), (ii),(iii) are obviously satisfied, while
 \begin{eqnarray*}
   |f_a(\sum\limits_{k\in D_a}\mu_ky_k')| &\le &
   \frac{1}{m_{2q}}\bigg(\big|f_{\beta_{l_0}}(\sum\limits_{k\in
   D_{\beta_{l_0}}}\mu_ky_k')\big|+|f_{\beta_{l_0}}(\mu_{k_0}y_{k_0}')|\\
   &  &   \qquad\qquad+\sum\limits_{k\in
   D_a}\sum\limits_{l=l_0+1}^{d}|f_{\beta_l}(\mu_ky_k')|\bigg)\\
      &\le  & \frac{1}{m_{2q}}\bigg( 72\cdot g_{\beta_{l_0}}(\sum\limits_{k\in
   D_{\beta_{l_0}}}\mu_kt_k)+h_{\beta_{l_0}}(\sum\limits_{k\in
   D_{\beta_{l_0}}}\mu_kt_k)\\
   & & \qquad\qquad+36|\mu_{k_0}|+\frac{1}{m_{2i-1}}\bigg)\\
   & \le & (72g_a+h_a)(\sum\limits_{k\in D_a}\mu_kt_k).
  \end{eqnarray*}
 The inductive construction is complete.

 For the  functionals $g_0,h_0$ corresponding
 to the root $0\in\mc{A}$ of the tree $\mc{A}$, and taking into account that $D_0=D$,
   we get that
  \begin{eqnarray*}
 |f(\sum\limits_{k\in D}\mu_ky_k')|& \le &
 72\cdot g_0(\sum\limits_{k\in D_0}\mu_kt_k)+ h_0 (\sum\limits_{k\in
 D_0}\mu_kt_k)\\
  &\le & 72\cdot
  g_0(\sum\limits_{k=1}^{2p_j}\mu_kt_k)+\max_{k}|\mu_k|\cdot
  2p_j \cdot\frac{1}{m_{2i-1}}  \\
  & \le &72 \cdot\|\sum\limits_{k=1}^{2p_j}\mu_kt_k\|_{J_{T_0}}+
 \max_k |\mu_k|\\
 &\le & 73\cdot \|\sum\limits_{k=1}^{2p_j}\mu_kt_k\|_{J_{T_0}}.
 \end{eqnarray*}
 The proof of the claim is complete.
 \end{proof}

 From Claim 1 and Claim 2 we conclude that
   \begin{eqnarray*}
   |f(\sum\limits_{k=1}^{2p_j}\mu_ky_k)|& \le &
   |f(\sum\limits_{k\in F}\mu_ky_k)|+|f(\sum\limits_{k\in
   D}\mu_ky_k')|+|f(\sum\limits_{k\in D}\mu_ky_k'')|\\
   & \le & 3\cdot\|\sum\limits_{k=1}^{2p_j}\mu_kt_k\|_{J_{T_0}}+
    73 \cdot\|\sum\limits_{k=1}^{2p_j}\mu_kt_k\|_{J_{T_0}}
    +73\cdot\|\sum\limits_{k=1}^{2p_j}\mu_kt_k\|_{J_{T_0}}\\
  &\le & 150\cdot \|\sum\limits_{k=1}^{2p_j}\mu_kt_k\|_{J_{T_0}}.
   \end{eqnarray*}
 This completes the proof of the right side inequality of \eqref{Neq1}
 and also the proof of Theorem \ref{Nth2}.
 \end{proof}

 \end{document}